\documentclass{amsart}
\usepackage{amssymb}
\usepackage{amsmath}
\usepackage{amsfonts}
\usepackage{chngcntr}
\newtheorem{theorem}{Theorem}[section]

\newtheorem{prop}[theorem]{Proposition}

\newtheorem{remark}[theorem]{Remark}

\newtheorem{innercustomthm}{Proposition}
\newenvironment{customthm}[1]
  {\renewcommand\theinnercustomthm{#1}\innercustomthm}
  {\endinnercustomthm}

\renewcommand{\l}{\lambda}
\newcommand{\R}{\mathbb R}

\newcommand{\p}{\partial}

\newcommand{\tA}{{\tilde{A}}}

\begin{document}
\title[Large data global regularity for the Faddeev model]{Large data global regularity for the $2+1$-dimensional equivariant Faddeev model}

\author{Dan-Andrei Geba and Manoussos G. Grillakis}

\address{Department of Mathematics, University of Rochester, Rochester, NY 14627, U.S.A.}
\email{dangeba@math.rochester.edu}
\address{Department of Mathematics, University of Maryland, College Park, MD 20742, U.S.A.}
\email{mggrlk@math.umd.edu}

\date{}

\begin{abstract}
This article addresses the large data global regularity for the equivariant case of the $2+1$-dimensional Faddeev model and shows that it holds true for initial data in $H^s\times H^{s-1}(\R^2)$ with $s>3$.
\end{abstract}

\subjclass[2000]{81T13}
\keywords{Faddeev model; Skyrme model; global solutions.}

\maketitle

%\tableofcontents

\section{Introduction}

%%%%%%%%%%%%%%%%%%%%%%%%%%%%%%%%%%%%%%%%%%%%%%%%%%%%%%%

\subsection{Statement of the problem and main result} 

An important classical field theory that models elementary heavy particles by topological solitons was proposed by Faddeev in \cite{F75, F76}. It is worth knowing that the Faddeev model admits knotted solitons. This theory is described by the action
\begin{equation}
S=  \int_{\R^{3+1}}\left\{\frac{1}{2}\partial_\mu {\bf n}\cdot \partial^\mu {\bf n}\, + \,\frac{1}{4}(\partial_\mu {\bf n}\wedge\partial_\nu {\bf n})\cdot(\partial^\mu {\bf n}\wedge\partial^\nu {\bf n})\right\}dg,
\label{fdv}
\end{equation}
where $\wedge$ denotes the usual cross product of vectors in $\mathbb{R}^3$ and ${\bf n}: \R^{3+1}\to \mathbb{S}^2$ are maps from the classical Minkowski spacetime, with $g= \text{diag}(-1,1,1,1)$, into the unit sphere of $\mathbb{R}^{3}$ endowed with the round metric. The associated Euler-Lagrange equations are given by  
\begin{equation}
{\bf n}\wedge\partial_\mu\partial^\mu {\bf n}+(\partial_\mu[{\bf n}\cdot(\partial^\mu {\bf n}\wedge\partial^\nu {\bf n})])\partial_\nu {\bf n}=0,
\label{fsys}
\end{equation}
which is a system of quasilinear wave equations. One can naturally extend this model by switching the domain of ${\bf n}$ from $\R^{3+1}$ to $\R^{n+1}$, which is also equipped with the Minkowski metric. 

The Faddeev model is intimately tied to the celebrated Skyrme model \cite{S1, S2, S3}, also known to be  the first classical theory modeling particles by topological solitons. The action for the Skyrme model is specified by
\begin{equation}
\label{sk}
S=\int_{\mathbb{R}^{3+1}} \left\{\frac 12 \langle\partial^\mu\phi, \partial_\mu\phi \rangle_h+\frac{\alpha^2}{4}\left( \langle\partial^\mu\phi, \partial_\mu\phi \rangle_h^2-\langle\partial^\mu\phi, \partial^\nu\phi \rangle_h \langle\partial_\mu\phi, \partial_\nu\phi \rangle_h \right)\right\}dg,
\end{equation}
where $\alpha$ is a constant having the dimension of length and $\phi: \R^{3+1}\to \mathbb{S}^3$ are maps from the $3+1$-dimensional Minkowski spacetime into the $3$-dimensional unit sphere. If one restricts the image of $\phi$ to be the equatorial 2-sphere of $\mathbb{S}^3$ (identified in this case with $\mathbb{S}^2$) by prescribing 
\begin{equation*}
\phi=(u, {\bf n})=(\pi/2, {\bf n}),
\end{equation*} 
where the round metric on $\mathbb{S}^3$ is 
\begin{equation*}
 h=du^2+\sin^2 u\,d{\bf n}^2,\qquad 0 \leq u\leq \pi, \qquad {\bf n}\in\mathbb{S}^2,
\end{equation*}
and sets $\alpha=1$, then the Skyrme action \eqref{sk} reduces to the Faddeev one \eqref{fdv}. We ask the interested reader to consult our monograph \cite{GGr} and references therein for a more comprehensive view on the physical descriptions and motivations for both models.

In our recent work \cite{GGr-17}, we studied the large data global regularity question for the equivariant case of the Skyrme model. For this paper, using a comparable approach, we plan to investigate the same issue corresponding to the Faddeev model. Our findings concern the $2+1$-dimensional equivariant version of this theory and this is because, in our opinion, this is the most natural setting in which to impose an equivariant ansatz on the Faddeev model. Further motivation for this choice will be provided in the next subsection. 

Thus, we work with maps ${\bf n}: (\mathbb{R}^{2+1},g)\to(\mathbb{S}^2,h)$ satisfying
\begin{equation*}
{\bf n}(t,r,\omega)\,=\,(u(t,r),\omega), \quad g= -dt^2+dr^2+r^2\,d\omega^2,\quad h=du^2+\sin^2u\,d\omega^2,
\end{equation*}
and solving \eqref{fsys}. It is straightforward to deduce that the only germane equation to be analyzed is the 
one for the azimuthal angle $u$,
\begin{equation}
\left(1+\frac{\sin^2u}{r^2}\right)(u_{tt}-u_{rr})-\left(1-\frac{\sin^2u}{r^2}\right)\frac{u_r}{r}+\frac{\sin2u}{2r^2}\left(1+u_t^2-u_r^2\right)=0, \label{main}
\end{equation}
which is of quasilinear type. Associated to this equation, there exists an a priori conserved energy given by\begin{equation}
E[u](t)=\int_0^\infty \left\{\left(1 + \frac{\sin ^2 u}{r^2}\right)\frac{u_t^2+u_r^2}{2}+
\frac{\sin ^2 u}{2r^2} 
\right\}\,r dr. \label{tote}
\end{equation}
We are interested in studying finite energy solutions, which necessarily obey 
\[
u(t,0)\,\equiv\, u(t,\infty)\,\equiv\,0\,(\text{mod}\,\pi).\]
The integer
\[
\frac{u(t,\infty)-u(t,0)}{\pi}
\]
is called the \textit{topological charge} of the map ${\bf n}$ and, like the energy, it is also conserved in time. Accordingly, we make the assumption 
\begin{equation}
u(t,0)=N_1\pi, \quad N_1\in \mathbb{N}, \qquad u(t,\infty)=0.
\label{bdry}
\end{equation}

The following theorem is the main result of this article. 
\begin{theorem}
Let $(u_0,u_1)$ be radial initial data with
\[
(u_0,u_1)\in H^s\times H^{s-1}(\R^2), \quad s>3,
\]
which meet the compatibility conditions
 \[u_0(0)=N_1\pi, \qquad u_0(\infty)=u_1(0)=u_1(\infty)=0.\] 
Then there exists a global radial solution $u$ to the Cauchy problem associated to \eqref{main} with $(u(0),u_t(0))=(u_0,u_1)$, satisfying \eqref{bdry} and  
\[
u\in C([0,T], H^s(\R^2))\cap C^1([0,T], H^{s-1}(\R^2)), \qquad (\forall)\,T>0.
\] 
\label{main-th}
\end{theorem}

\begin{remark}
This result should be compared to what is known about the $2+1$-dimensional equivariant wave map equation, i.e.,
\begin{equation*}
u_{tt}-u_{rr}-\frac{u_r}{r}+\frac{\sin2u}{2r^2}=0,
\end{equation*}
which is of semilinear type and for which \eqref{main} could be seen as a quasilinear generalization. It may come as a surprise to learn that there are smooth data that lead to finite time collapse for solutions of this equation. We refer the reader to work by Rapha\"{e}l and Rodnianski \cite{RR} and references therein. 
\end{remark}
%%%%%%%%%%%%%%%%%%%%%%%%%%%%%%%%%%%%%%%%%%%%%%%%%%%%%%%

\subsection{Comments on previous relevant works and comparison to main result} 

Likely due to the intricate nature of the variational system \eqref{fsys}, initial investigations into the Faddeev model concentrated on its static properties. Faddeev and Niemi \cite{FN97} and Battye and Sutcliffe \cite{BS98, BS99} performed numerical simulations for various topological solitons, while Vakulenko and Kapitanski \cite{VK79} and Lin and Yang \cite{LY104, LY04} investigated the associated topologically-constrained energy-minimization problem. Further references on the static problem can be found in the excellent book by Manton and Sutcliffe \cite{MS}. 

For the time-dependent case, the most natural issue to study about either \eqref{fsys} or \eqref{main} is the well-posedness of the associated Cauchy problem. This is a very difficult enterprise, owing to the quasilinear nature of the equations in question and to the fact that the initial value problem is supercritical with respect to the energy (for more details, see \cite{GNZ}). 

To our knowledge, the first result concerning the evolution problem belongs to Lei, Lin, and Zhou \cite{LLZ11}, who showed that the $2+1$-dimensional system \eqref{fsys} is globally well-posed for smooth, compactly-supported initial data with small $H^{11}(\R^2)$ norm. This was followed by work of Geba, Nakanishi, and Zhang \cite{GNZ} in the equivariant case, which consists of global well-posedness and scattering for\eqref{main}, with $N_1=0$ and initial data having a small Besov-Sobolev norm at the level of $H^{2}(\R^2)$. 

Nevertheless, our paper is mostly related to an article by Creek \cite{C13}, who proved Theorem \ref{main-th} under the more constrained assumption $s\geq 4$. This is achieved by adapting a framework due to Li \cite{Li}, used in demonstrating a similar result for the Skyrme model. Li's approach also influenced our recent work \cite{GGr-17}, also on the Skyrme model, in which a result similar to Theorem \ref{main-th}
was proved for that particular problem. In the current paper, we try to emulate the argument in \cite{GGr-17} and, overall, the analysis follows along the same lines, being, in fact, more direct for certain steps. Nevertheless, there are a number of instances where we face challenges not present in the proof for the Skyrme model and we have to come up with novel ways to handle them. Throughout the paper, we make numerous remarks regarding the similarities and discrepancies between the argument in \cite{GGr-17} and the current one.

We conclude these comments by pointing out two more facts. First, the main difference between the present work  and Creek's is that our approach is able to handle fractional derivatives. Secondly, we believe our result is optimal when it comes to the tools used in its proof; however, we do not follow up on this issue here.

%%%%%%%%%%%%%%%%%%%%%%%%%%%%%%%%%%%%%%%%%%%%%%%%%%%%%%%

\subsection{Outline of the paper}

In the next section, the main result is reformulated in terms of a newly introduced function $v$ and its argument is reduced to the verification of a continuation criterion for $v$. Following this, we insert another auxiliary function $\Phi$, which is closely related to $v$, but is more tractable to the techniques we plan to use. In the same section, we perform one final reduction that leaves us to argue for the finiteness of certain Sobolev norms for derivatives of $\Phi$. For section 3, we gather the necessary notational conventions and the analytic toolbox relied upon throughout the article. Our analysis starts in section 4, where we prove energy-type bounds for both $v$ and $\Phi$, which yield preliminary fixed-time decay estimates. The subsequent two sections are devoted to upgrading this information to the level of $H^2$ and $H^3$ regularities, respectively. In section 7, we finish the argument by proving that $\Phi$ has just enough regularity to force the validity of the continuation criterion for $v$. We conclude the paper by including an appendix that confirms the Sobolev regularity required of various initial data in certain steps of the proof.

%%%%%%%%%%%%%%%%%%%%%%%%%%%%%%%%%%%%%%%%%%%%%%%%%%%%%%%

\subsection*{Acknowledgements}

The first author was supported in part by a grant from the Simons Foundation $\#$ 359727.

%%%%%%%%%%%%%%%%%%%%%%%%%%%%%%%%%%%%%%%%%%%%%%%%%%%%%%%

\section{Preliminaries}

%%%%%%%%%%%%%%%%%%%%%%%%%%%%%%%%%%%%%%%%%%%%%%%%%%%%%%%

\subsection{Introducing the function $v$ and initial reductions}

First, we recast the equation \eqref{main} as
\begin{equation}
\Box_{2+1}u= N(r,u,\nabla u), \label{main-2}
\end{equation}
with
\[
N(r,u,\nabla u):=\frac{-\frac{\sin2u}{2r^2}\left(1+u_t^2-u_r^2\right)}{1 + \frac{\sin ^2 u}{r^2}}-\frac{\frac{2\sin^2u}{r^3}u_r}{1 + \frac{\sin ^2 u}{r^2}}
\]
and
\[
\Box_{2+1}=\p_{tt}-\p_{rr}-\frac{1}{r}\p_r
\]
being the radial wave operator in $\R^{2+1}$. We perform the substitution
\begin{equation}
u(t,r)\,=\,r\,v(t,r)+\varphi(r),
\label{utov}
\end{equation}
where $\varphi:\R_+\to\R_+$ is a smooth, decreasing function, satisfying $\varphi\equiv N_1\pi$ on $[0,1]$ and 
$\varphi\equiv 0$ on $[2,\infty)$. We also need to insert a finer version of $\varphi$, denoted $\varphi_{<1}$, which has the same smoothness and monotonicity as $\varphi$, but now obeys $\varphi_{<1}\equiv 1$ on $[0,1/2]$ and $\varphi_{<1}\equiv 0$ on $[1,\infty)$. Moreover, the function $1-\varphi_{<1}$ is labelled $\varphi_{>1}$. Consequently, we derive that
\begin{equation}
\aligned
\Box_{4+1}v=&\frac{1}{r}\Delta_2\varphi+\frac{1}{r}\varphi_{>1}N(r, rv+\varphi, \nabla(rv+\varphi))+\frac{1}{r^2}\varphi_{>1}v\\
&+\varphi_{<1}\left(\frac{1}{r}N(r, rv, \nabla(rv))+\frac{1}{r^2}v\right)
\endaligned
\label{main-v}
\end{equation}
and 
\begin{equation}
\aligned
&\frac{1}{r}N(r, rv, \nabla(rv))+\frac{1}{r^2}v\\
&\qquad\qquad=\frac{1}{1+N_0(rv)v^2}\bigg\{N_1(rv)v^3+N_2(rv)v^5+N_3(rv)v(v_t^2-v_r^2)\\
&\qquad\qquad\qquad\qquad\qquad\qquad+2N_2(rv)rv^4v_r\bigg\},
\endaligned
\label{Nv-l1}
\end{equation}
with all $N_i=N_i(x)$ being even, analytic, and satisfying
\begin{equation}
\|\partial^kN_i\|_{L^\infty(\R)}\leq C_k, \qquad (\forall)\,k\in\mathbb{N}.
\label{dni-li}
\end{equation}
The reader can find the precise formulae for these functions in \cite{C13}.

The motivation for the substitution \eqref{utov} is that the proof of Theorem \ref{main-th} is reduced to the one of the following result, a fact that could be verified in a direct manner (e.g., see Subsection 2.3 in \cite{C13}).

\begin{theorem}
Let $(v_0,v_1)$ be radial initial data with
\[
(v_0,v_1)\in H^s\times H^{s-1}(\R^4), \quad s>3.
\]
Then there exists a global radial solution $v$ to the Cauchy problem associated to \eqref{main-v} with $(v(0),v_t(0))=(v_0,v_1)$, which satisfies 
\[
v\in C([0,T], H^s(\R^4))\cap C^1([0,T], H^{s-1}(\R^4)), \qquad \forall\,T>0.
\] 
\label{main-th-v}
\end{theorem}
\noindent To prove this theorem, we use a classical result (e.g., see Theorem 6.4.11 in H\"{o}rmander \cite{H97})) which allows us to obtain global solutions from local ones, which additionally satisfy a continuation criterion. Thus, the entire argument is reduced to demonstrating the next theorem.

\begin{theorem}
For any $0<T<\infty$ and $s>3$, a radial solution $v$ on $[0,T)$ to \eqref{main-v} with $(v(0),v_t(0))\in H^s\times H^{s-1}(\R^4)$  satisfies 
\begin{equation}
\|(1+r)(|v|+|\nabla_{t,x}v|)\|_{L_{t,x}^\infty([0,T)\times \R^4)}\,<\,\infty.
\label{livr}
\end{equation}
\label{main-th-v-2}
\end{theorem}

%%%%%%%%%%%%%%%%%%%%%%%%%%%%%%%%%%%%%%%%%%%%%%%%%%%%

\subsection{The construction of the auxiliary function $\Phi$ and further reductions}

We prove Theorem \ref{main-th-v-2} in quite a roundabout way, by showing that \eqref{livr} holds true using a newly constructed function $\Phi$. This satisfies an equation which is easier to study than  \eqref{main-v}. The first step in constructing $\Phi$ is directed at the derivative terms on the right-hand side of \eqref{main-v} and, for that purpose, we take
\[
\Phi_1(t,r)=\int^{u(t,r)}_{N_1\pi}\left(1+\frac{\sin^2w}{r^2}\right)^{1/2}dw,
\]
which satisfies the wave equation
\[
\Box_{2+1}\Phi_1=-\frac{1}{r^2}\Phi_1\,+\,\int^{u}_{N_1\pi}\left\{A^{1/2}-A^{-3/2}\right\}dw,
\] 
with
\[
A=A(r,w):=1+\frac{\sin^2w}{r^2}.
\]

Next, we handle the $1/r^2$ singularity by introducing
\[
\Phi_2(t,r)=\frac{\Phi_1(t,r)}{r},
\]
which solves
\[
\Box_{4+1}\Phi_2=\Phi_2\,-\,\frac{1}{r}\int^{u}_{N_1\pi}\left\{A^{-3/2}\right\}dw.
\]
Apparently, it seems that we have a new singularity in front of the integral to deal with. In fact, one can see that it is removable by writing
\[
\frac{1}{r}=\frac{\varphi_{<1}}{r}+\frac{\varphi_{>1}}{r} 
\]
and then making the  change of variable 
\[
w=N_1\pi+ry
\] 
in the integral multiplied by $\varphi_{<1}$. We need to make one more adjustment and this is because a formal calculation shows that we might have
\[
\|\Phi_2\|_{L^2(\{r\geq 1\})}=\infty,
\]
which is not what we expect from our approach.

We take care of this final issue by choosing
\[
\Phi=\Phi_2\,+\,\frac{1}{r}\varphi_{>1}\int_{0}^{N_1\pi}\left\{A^{-3/2}\right\}dw,
\]
which leads after careful computations to 
\begin{equation}
\Phi\,=\,\int_{0}^{v}\left(1+\frac{\sin^2(ry+\varphi)}{r^2}\right)^{1/2}dy\,+\,\frac{\varphi_{\geq 1/2}}{r^3}.
\label{Phi-final}
\end{equation}
The corresponding wave equation is
\begin{equation}
\Box_{4+1}\Phi=\Phi\,-\,\int_{0}^{v}\left\{\tA^{-3/2}\right\}dy\,+\,\frac{\varphi_{\geq 1/2}}{r^3},
\label{Box-Phi}
\end{equation}
where
\begin{equation}
\tA=\tA(r,y):=1+\frac{\sin^2(ry+\varphi(r))}{r^2}
\label{ta-formula}
\end{equation}
and $\varphi_{\geq 1/2}=\varphi_{\geq 1/2}(r)$ is a generic smooth function, with bounded derivatives of all orders and supported in the domain $\{r\geq 1/2\}$, which may change from line to line.

In order to prove Theorem \ref{main-th-v-2}, it is clear that we can additionally assume, without loss of generality, that $s$ is sufficiently close to $3$. In fact, we argue that if $3<s<4$, then
\begin{equation}
\aligned
\|\Phi\|_{L^\infty H^{s}([0,T)\times \R^4)}&+\|\Phi_t\|_{L^\infty H^{s-1}([0,T)\times \R^4)}+\|\Phi_{tt}\|_{L^\infty H^{s-2}([0,T)\times \R^4)}\\
&+\|\Phi_{ttt}\|_{L^\infty L^2([0,T)\times \R^4)}\,<\,\infty.
\endaligned
\label{Phi-li}
\end{equation}
Jointly with Sobolev embeddings and radial Sobolev inequalities, this estimate implies \eqref{livr}.

%%%%%%%%%%%%%%%%%%%%%%%%%%%%%%%%%%%%%%%%%%%%%%%%%%%%

\section{Notations and analytic toolbox}

%%%%%%%%%%%%%%%%%%%%%%%%%%%%%%%%%%%%%%%%%%%%%%%%%%%%

\subsection{Notational conventions}

First, we write $A\lesssim B$ to designate $A\leq CB$, where $C$ is a constant depending only upon parameters that are considered fixed throughout the paper. Two such parameters are the conserved energy \eqref{tote}, expressed in terms of the initial data $(u_0, u_1)$ in Theorem \ref{main-th} as
\begin{equation}
E:=\int_0^\infty \left\{\left(1 + \frac{\sin ^2 u_0}{r^2}\right)\frac{u_1^2+u_{0,r}^2}{2}+
\frac{\sin ^2 u_0}{2r^2} 
\right\}r dr,
\label{tote-0}
\end{equation}
and the time $0<T<\infty$ featured in Theorem \ref{main-th-v-2}. We write $A\sim B$ for the case when both $A\lesssim B$ and $B\lesssim A$ are valid.

Secondly, for the function $w=w(t,x)$, we work with $\nabla w=(\partial_t w, \nabla_x w)$ and
\begin{equation*}
\|w\|_{L^pX(I\times \R^n)}=\|w\|_{L_t^pX_x(I\times \R^n)}= \left(\int_I\|w(t,\cdot)\|_{X(\R^n)}^{p}dt\right)^{1/p},
\end{equation*}
where $X(\R^n)$ is a normed/semi-normed space (e.g., $X=L^q$ or $H^\sigma$ or $\dot{H}^\sigma$) and $I\subseteq \R$ is an arbitrary time interval. For ease of notation, when $I\times \R^n=[0,T)\times \R^4$, we simply write
\begin{equation*}
\|w\|_{L^pX}=\|w\|_{L^pX([0,T)\times \R^4)}.
\end{equation*}
This has to do with the majority of the norms we are dealing with from here on out referring to this particular situation.

%%%%%%%%%%%%%%%%%%%%%%%%%%%%%%%%%%%%%%%%%%%%%%%%%%%%

\subsection{Analytic toolbox}

In here, we collect a list of analytic facts that are used throughout the argument. First, we recall the classical and general Sobolev embeddings
\begin{align}
H^\sigma(\R^n)\subset L^\infty(\R^n),\qquad \sigma>\frac{n}{2}&, \label{Sob-classic}\\
\dot{H}^{\sigma,p}(\R^n)\subset L^q(\R^n),\qquad 1<p\leq q<\infty, \quad &\sigma=n\left(\frac{1}{p}-\frac{1}{q}\right),  \label{Sob-gen}
\end{align}
and the radial Sobolev estimates (\cite{St77}, \cite{CO09})
\begin{align}
r^{n/2-\sigma}|&f(r)|\lesssim \|f\|_{\dot{H}^\sigma(\R^n)},\qquad \frac{1}{2}<\sigma<\frac{n}{2}, \label{rad-Sob-1}\\
&r^{(n-1)/2}|f(r)|\lesssim \|f\|_{H^1(\R^n)},  \label{rad-Sob-2}
\end{align}
which are valid for radial functions defined on $\R^n$. In connection to these, we write down Hardy's inequality (\cite{OK90})
\begin{equation}
\left\|\frac{g}{|x|}\right\|_{L^p(\R^n)}\lesssim \left\|\nabla_x g\right\|_{L^p(\R^n)}, \qquad 1<p<n, \label{Hardy}
\end{equation}
and the interpolation bound (\cite{MR0482275})
\begin{equation}\aligned
&\begin{cases}
\sigma_1\neq \sigma_2, \qquad 1\leq p,p_1,p_2\leq \infty, \qquad 0\leq \theta \leq 1,\\ 
\sigma=(1-\theta)\sigma_1+\theta\sigma_2,\qquad \frac{1}{p}=\frac{1-\theta}{p_1}+\frac{\theta}{p_2},
\end{cases}\\
&\qquad\|g\|_{\dot{H}^{\sigma,p}(\R^n)}\lesssim \|g\|^{1-\theta}_{\dot{H}^{\sigma_1,p_1}(\R^n)}\|g\|^\theta_{\dot{H}^{\sigma_2,p_2}(\R^n)},
\endaligned
\label{interpol-bd}
\end{equation}
both of which hold true for general functions on $\R^n$.

Next, we use the Riesz potential $D^\sigma=(-\Delta)^{\sigma/2}$ to record the fractional Leibniz estimate (\cite{GO-14}, \cite{BLi-14})
\begin{equation}
\aligned
&\sigma>0, \quad 1\leq p\leq \infty, \quad 1< p_1, p_2, q_1, q_2\leq \infty, \quad\frac{1}{p}=\frac{1}{p_1}+\frac{1}{q_1}=\frac{1}{p_2}+\frac{1}{q_2}, \\
&\,\|D^\sigma(fg)\|_{L^p(\R^n)}\lesssim \|D^\sigma f\|_{L^{p_1}(\R^n)}\|g\|_{L^{q_1}(\R^n)}+\|f\|_{L^{p_2}(\R^n)}\|D^\sigma g\|_{L^{q_2}(\R^n)}, 
\endaligned
\label{Lbnz-0}
\end{equation}
and the Kato-Ponce type inequalities (\cite{Li-16})
\begin{equation}
\aligned
0<\sigma<2, \quad 1< p, p_1, p_2 &<\infty, \quad \frac{1}{p}=\frac{1}{p_1}+\frac{1}{p_2},\\
\|D^\sigma(fg)-D^\sigma f\, g-f\,D^\sigma g\|_{L^p(\R^n)}&\lesssim \|D^{\sigma/2} f\|_{L^{p_1}(\R^n)}\|D^{\sigma/2}g\|_{L^{p_2}(\R^n)}, 
\endaligned
\label{Lbnz-1}
\end{equation}
\begin{equation}
\aligned
&0<\sigma\leq 1, \qquad 1< p <\infty,\\
\|D^\sigma(fg)-f\,&D^\sigma g\|_{L^p(\R^n)}\lesssim \|D^{\sigma} f\|_{L^{p}(\R^n)}\|g\|_{L^{\infty}(\R^n)}.
\endaligned
\label{Lbnz-2}
\end{equation}
We also call to mind the well-known Moser bound
\begin{equation}
\|F(f)\|_{H^\sigma(\R^n)}\leq \gamma(\|f\|_{L^{\infty}(\R^n)})\,\|f\|_{H^{\sigma}(\R^n)}, \qquad(\forall)\,f\in L^{\infty}\cap H^\sigma(\R^n; \R^k), 
\label{Moser}
\end{equation}
where $F\in C^\infty(\R^k;\R)$, $F(0)=0$, and $\gamma=\gamma(\sigma)\in C(\R; \R)$. Following this, we recall the Bernstein estimates
\begin{equation}
\aligned
\|P_{>\l}f\|&_{L^p(\R^n)}\lesssim \|f\|_{L^p(\R^n)},\\ \l^\sigma\|P_{>\l}f\|_{L^p(\R^n)}\lesssim& \,\|P_{>\l}D^\sigma f\|_{L^p(\R^n)}, \qquad \sigma\geq 0,
\endaligned
\label{Bernstein}
\end{equation}
where $P_{>\l}$ is a Fourier multiplier localizing the spatial frequencies to the region $\{|\xi|>\l\}$.

Finally, we recount the classical Strichartz inequalities for the $4+1$-dimensional linear wave equation, which take the form
\begin{equation}
\aligned
\|\Psi\|&_{L^pL^q(I\times \R^4)}+\|\Psi\|_{L^\infty\dot{H}^\sigma(I\times \R^4)}+\|\Psi_{t}\|_{L^\infty\dot{H}^{\sigma-1}(I\times \R^4)}\\
&\qquad\lesssim \|\Psi(0)\|_{\dot{H}^\sigma( \R^4)}+\|\Psi_{t}(0)\|_{\dot{H}^{\sigma-1}(\R^4)}+\|\Box\Psi\|_{L^{\bar{p}'}L^{\bar{q}'}(I\times \R^4)},
\endaligned
\label{Str-gen} 
\end{equation}
with $I$ being a time interval and
\[
\begin{cases}
4\leq p\leq \infty, \qquad 2\leq q< \infty, \qquad \frac{4}{p}+\frac{2}{q}\leq 1, \\
1\leq \bar{p}' \leq \frac{4}{3}, \qquad\, 1<\bar{q}' \leq 2, \qquad \frac{4}{\bar{p}'}+\frac{2}{\bar{q}'}\geq 5, \\
\frac{1}{p}+\frac{4}{q}= 2-\sigma= -2+\frac{1}{\bar{p}'}+\frac{4}{\bar{q}'}.
\end{cases}
\]
A straightforward consequence of the previous bound is the following generalized energy estimate:
\begin{equation}
\aligned
\|\Psi\|_{L^\infty\dot{H}^\sigma(I\times \R^4)}+\|\Psi_{t}\|_{L^\infty\dot{H}^{\sigma-1}(I\times \R^4)}\lesssim &\|\Psi(0)\|_{\dot{H}^\sigma( \R^4)}+\|\Psi_{t}(0)\|_{\dot{H}^{\sigma-1}(\R^4)}\\
&+\|\Box\Psi\|_{L^1\dot{H}^{\sigma-1}(I\times \R^4)}.
\endaligned
\label{en-hs} 
\end{equation}

%%%%%%%%%%%%%%%%%%%%%%%%%%%%%%%%%%%%%%%%%%%%%%%%%%%%

\section{Energy-type arguments} \label{sect-en}

In this section, we truly start the argument by showing that
\begin{equation}
\|v\|_{L^\infty H^{1}}+\|v_t\|_{L^\infty L^{2}}\lesssim 1
\label{v-h1}
\end{equation}
and, subsequently,
\begin{equation}
\|\Phi\|_{L^\infty H^{1}}+\|\Phi_t\|_{L^\infty L^{2}}\lesssim 1.
\label{Phi-h1}
\end{equation}
Next, we use these bounds and the radial Sobolev inequalities \eqref{rad-Sob-1}-\eqref{rad-Sob-2} to derive preliminary fixed-time decay estimates for both $v$ and $\Phi$, which, in turn, yield valuable asymptotics for $\Phi$ and $\Box\Phi$.

\vspace{.1in}
\hrule

%%%%%%%%%%%%%%%%%%%%%%%%%%%%%%%%%%%%%%%%%%%%%%%%%%%

\subsection{Energy-type arguments for $u$} 

Let $I:\R\to\R$ be defined by
\begin{equation*}
I(w):=\int_0^w|\sin z|\,dz,
\end{equation*}
which is easily seen to be odd, strictly increasing, and satisfying
\begin{equation*}
\lim_{|w|\to\infty} |I(w)|=\infty.
\end{equation*}
If we rely on \eqref{bdry}, the fundamental theorem of calculus, the Cauchy-Schwarz inequality, and \eqref{tote}, it follows that
\begin{equation*}
\aligned
I(|u(t,r)-u(t,0)|)&=|I(u(t,r))-I(u(t,0)))|\\
&\leq \int_0^r|\sin(u(t,s))|\,|u_r(t,s)|\,ds\\ 
&\lesssim \min\bigg\{\int_0^r\frac{\sin^2(u(t,s))}{s^2}\,s\,ds \cdot\int_0^ru^2_r(t,s)\,s\,ds,\\
&\qquad\qquad\int_0^r\frac{\sin^2(u(t,s))u^2_r(t,s)}{s^2}\,s\,ds\cdot\int_0^r s\,ds\bigg\}^{1/2}\\
&\lesssim \min\left\{E, E^{1/2}r\right\}.
\endaligned
\end{equation*}
Therefore, by taking into account the properties of $I$, we first deduce
\begin{equation*}
|u(t,r)-u(t,0)|\lesssim 1,
\end{equation*}
while for $r$ sufficiently small, we can be more precise and write
\begin{equation*}
|u(t,r)-u(t,0)|<\frac{\pi}{2}.
\end{equation*}
Both of these estimates are uniform in time. Based on the definition of $I$, the latter implies
\begin{equation*}
|u(t,r)-u(t,0)|^2\sim I(|u(t,r)-u(t,0)|)\lesssim E^{1/2}r,
\end{equation*}
which leads to  
\begin{equation*}
|u(t,r)-u(t,0)|\lesssim r^{1/2}.
\end{equation*}
Thus, we can summarize these findings as
\begin{equation}
|u(t,r)-u(t,0)|\lesssim \min\left\{1, r^{1/2}\right\}.
\label{decay-u}
\end{equation}

%%%%%%%%%%%%%%%%%%%%%%%%%%%%%%%%%%%%%%%%%%%%%%%%%%%

\subsection{Energy-type arguments for $v$} 

Using the formula \eqref{utov}, we easily infer that 
\begin{equation}
\|v_t\|_{L^\infty L^2}\simeq\left\|u_t\right\|_{L^\infty L^2([0,T)\times \R^2)}\lesssim E^{1/2}\lesssim 1
\label{vt-l2}
\end{equation}
and, consequently,
\begin{equation}
\|v\|_{L^\infty L^2}\lesssim\|v(0)\|_{L^2(\R^4)}+T\,E^{1/2}  \lesssim 1.
\label{v-l2}
\end{equation}
For the radial derivative of $v$, the joint application of \eqref{utov}, \eqref{decay-u}, and \eqref{v-l2} produces
\begin{equation}
\aligned
\|v_r\|_{L^\infty L^2}&\lesssim\left\|\frac{\varphi_r}{r}\right\|_{L^2(\R^4)}+\left\|u_r\right\|_{L^\infty L^2([0,T)\times \R^2)}+\left\|\frac{u-\varphi}{r^2}\right\|_{L^\infty L^2}\\
&\lesssim 1+E^{1/2}+\left\|\frac{u(t,r)-u(t,0)}{r^2}\right\|_{L^\infty L^2([0,T)\times\{r\ll1\})}\\
&\quad+\left\|\frac{1}{r^2}\right\|_{L^\infty L^2([0,T)\times\{r\sim 1\})}+\left\|\frac{u}{r^2}\right\|_{L^\infty L^2([0,T)\times\{r\gg 1\})}\\
&\lesssim 1+E^{1/2}+\left\|\frac{1}{r^{3/2}}\right\|_{L^\infty L^2([0,T)\times\{r\ll 1\})}+\left\|\frac{v}{r}\right\|_{L^\infty L^2([0,T)\times \{r\gg 1\})}\\ 
&\lesssim 1+E^{1/2}\\ &\lesssim 1,
\endaligned
\label{vr-l2}
\end{equation}
which finishes the proof of \eqref{v-h1}.

\begin{remark}
In \cite{GGr-17},  the analysis for the Skyrme model relied on \eqref{Hardy} to derive
\[\aligned
\left\|\frac{u-\varphi}{r^2}\right\|_{L^\infty L^2([0,T)\times \R^5)}&\lesssim 1+\left\|\frac{u}{r}\right\|_{L^\infty L^2([0,T)\times \R^3)}\\&\lesssim 1+\left\|u_r\right\|_{L^\infty L^2([0,T)\times \R^3)}\\&\lesssim 1+E^{1/2}\\&\lesssim 1.
\endaligned
\]
Here, we really need the decay estimate \eqref{decay-u} and \eqref{v-l2} in order to bound 
\[
\left\|\frac{u-\varphi}{r^2}\right\|_{L^\infty L^2},
\]
as Hardy's inequality is inapplicable.
\end{remark}

%%%%%%%%%%%%%%%%%%%%%%%%%%%%%%%%%%%%%%%%%%%%%%%%%

\subsection{Energy-type arguments for $\Phi$} 

We take advantage of the formula \eqref{Phi-final} to deduce
\begin{equation}
\|\Phi_t\|_{L^\infty L^2}\simeq\left\|\left(1+\frac{\sin^2 u}{r^2}\right)^{1/2}u_t\right\|_{L^\infty L^2([0,T)\times \R^2)}\lesssim E^{1/2} \lesssim 1.
\label{phitlil2}
\end{equation}
Furthermore, another application of the same formula leads to 
\[\aligned
|\Phi(0)|\lesssim \int_0^{|v(0)|}(1+y)dy+\frac{|\varphi_{\geq 1/2}|}{r^3}\lesssim |v(0)|+|v(0)|^2+\frac{|\varphi_{\geq 1/2}|}{r^3},
\endaligned
\]
which, coupled with the Sobolev embeddings \eqref{Sob-gen}, yields 
\[
\|\Phi(0)\|_{L^2(\R^4)}\,\lesssim\,1+\|v(0)\|_{L^2(\R^4)}+\|v(0)\|^2_{L^4(\R^4)}\,\lesssim\,1+\|v(0)\|_{H^{1}(\R^4)}^2 \lesssim 1.
\]
As we argued for $v$, it follows that
\begin{equation}
\|\Phi\|_{L^\infty L^2}\lesssim 1+\|v(0)\|^2_{H^{1}(\R^4)}+T\,E^{1/2} \lesssim 1.
\label{philil2}
\end{equation}
Hence, in order to finish the proof of \eqref{Phi-h1}, we need to obtain a favorable bound for $\|\Phi_r\|_{L^\infty L^2}$, which is slightly more intricate. 

First, we show that the following fixed-time estimate is valid.
\begin{prop}
Under the assumptions of Theorem \ref{main-th-v-2},
\begin{equation}
\left|\partial_t\left\{\int_{\R^4}\left|\nabla \Phi\right|^2 dx\right\}\right|\lesssim 1
\label{phir}
\end{equation}
holds true uniformly in time on $[0,T)$.
\end{prop}
\begin{proof}
If we multiply the equation \eqref{Box-Phi} by $\Phi_t$ and integrate the outcome with respect to the spatial variables using Gauss's theorem, then we infer that
\begin{equation*}
\int_{\R^4}\big\{\Phi_t\,\Box_{4+1}\Phi\big\}dx=\partial_t\left\{\int_{\R^4}\left|\nabla \Phi\right|^2 dx\right\}.
\end{equation*}
Next, according to \eqref{Phi-final} and \eqref{Box-Phi}, we can write
\begin{equation}
\left|\Box_{4+1}\Phi\right|\lesssim |\Phi|+\frac{|\varphi_{\geq 1/2}|}{r^3}.
\label{boxphi-phi}
\end{equation}
Consequently, by also factoring in \eqref{phitlil2} and \eqref{philil2}, we derive
\begin{equation*}
\left|\int_{\R^4}\big\{\Phi_t\,\Box_{4+1}\Phi\big\}dx\right|\lesssim \|\Phi_t\|_{L^\infty L^2}\left(\|\Phi\|_{L^\infty L^2}+\left\|\frac{\varphi_{\geq 1/2}}{r^3}\right\|_{L^2(\R^4)}\right)\lesssim 1,
\end{equation*}
which gives the desired conclusion.
\end{proof}

On the basis of this result and \eqref{phitlil2}, it follows that 
\begin{equation}
\|\Phi_r\|_{L^\infty L^2} \lesssim 1
\label{phirlil2}
\end{equation}
is valid if we prove that
\begin{equation*}
\|\Phi_r(0)\|_{L^2(\R^4)} \lesssim 1.
\end{equation*}
For this purpose, a straightforward calculation using \eqref{Phi-final} and \eqref{ta-formula} yields
\[
\aligned
\Phi_r(0)\,=\,\tA^{1/2}(r,v(0))v_r(0)+\frac{1}{2}\int_{0}^{v(0)}\left\{\tA^{-1/2}\tA_r\right\}dy +\frac{\varphi_{\geq 1/2}}{r^3},
\endaligned
\]
with
\begin{equation}
\tA_r=\frac{-2\sin^2(ry+\varphi)}{r^3}+\frac{\sin2(ry+\varphi)\cdot(y+\varphi_r)}{r^2}.
\label{ta-r}
\end{equation}
If we rely on the properties of $\varphi$, then it is relatively easy to deduce, eventually applying Maclaurin series, that
\begin{equation}
1\leq \tA(r,v(0))\lesssim 1+v^2(0),
\end{equation}
\begin{equation}
\left|\frac{-2\sin^2(ry+\varphi)}{r^3}+\frac{\sin2(ry+\varphi)\cdot(y+\varphi_r)}{r^2}\right|\lesssim\frac{1+|y|}{r^2}, \qquad r\geq 1,
\label{ta-r-g1}
\end{equation}
\begin{equation}
\left|\frac{-2\sin^2(ry+\varphi)}{r^3}+\frac{\sin2(ry+\varphi)\cdot(y+\varphi_r)}{r^2}\right|\lesssim ry^4, \qquad r<1.
\label{ta-r-l1}
\end{equation}
Based on the last five mathematical statements and also using the Sobolev embeddings \eqref{Sob-gen}, we infer that
\begin{equation}
\aligned
&\|\Phi_r(0)\|_{L^2(\R^4)}\\
&\qquad\lesssim \left\|(1+|v(0)|)v_r(0)\right\|_{L^2(\R^4)}+\left\|\frac{v(0)}{r^2}\right\|_{L^2(\{r\geq 1\})}+\left\|\frac{v(0)}{r}\right\|^2_{L^4(\{r\geq 1\})}\\
&\qquad\quad+\left\|rv^5(0)\right\|_{L^2(\{r<1\})}+\left\|\frac{\varphi_{\geq 1/2}}{r^3}\right\|_{L^2(\R^4)}\\
&\qquad\lesssim\|v(0)\|_{\dot{H}^{1}(\R^4)}+\|v(0)\|_{L^8(\R^4)}\|v_r(0)\|_{L^{8/3}(\R^4)}+\|v(0)\|_{L^2(\R^4)}\\
&\qquad\quad+\|v(0)\|^2_{H^{1}(\R^4)}+\left\|rv^5(0)\right\|_{L^2(\{r<1\})}+1\\
&\qquad\lesssim 1+\|v(0)\|^2_{H^{3/2}(\R^4)}+\left\|rv^5(0)\right\|_{L^2(\{r<1\})}\\
&\qquad\lesssim 1+\left\|rv^5(0)\right\|_{L^2(\{r<1\})}.
\endaligned
\label{phir-l2}
\end{equation}
For the last norm, due to \eqref{utov} and \eqref{decay-u}, we obtain that
\begin{equation}
|v(t,r)|\lesssim \min\left\{\frac{1}{r}, \frac{1}{r^{1/2}}\right\}
\label{decay-v-0}
\end{equation}
and, subsequently,
\[
\left\|rv^5(0)\right\|_{L^2(\{r<1\})}\lesssim \left\|\frac{1}{r^{3/2}}\right\|_{L^2(\{r<1\})}\lesssim 1.
\]
This finishes the argument for \eqref{phirlil2} and, consequently, \eqref{Phi-h1}.

\begin{remark}
By comparison to the corresponding analysis in \cite{GGr-17}, the argument here is much more streamlined, mainly due to \eqref{boxphi-phi}. One can interpret this fact as the equivariant Faddeev equation \eqref{Box-Phi} behaving better than its Skyrme counterpart.
\end{remark}

%%%%%%%%%%%%%%%%%%%%%%%%%%%%%%%%%%%%%%%%%%%%%%%%%%%% 

\subsection{Preliminary decay estimates and asymptotics}

First, we work with \eqref{v-h1} and \eqref{Phi-h1} to derive fixed-time decay estimates for both $\Phi$ and $v$. 
\begin{prop}
Under the assumptions of Theorem \ref{main-th-v-2}, we have
\begin{align}
|\Phi(t,r)|&\lesssim \min\left\{\frac{1}{r^{3/2}}, \frac{1}{r}\right\},\label{decay-Phi}\\
|v(t,r)|&\lesssim \min\left\{\frac{1}{r^{3/2}}, \frac{1}{r^{1/2}}\right\}.\label{decay-v}
\end{align}
\end{prop}
\begin{proof}
By virtue of the radial Sobolev inequalities \eqref{rad-Sob-1} and \eqref{rad-Sob-2}, we deduce \begin{equation*}
\aligned
|\Phi(t,r)|&\lesssim \min\left\{\frac{1}{r^{3/2}}, \frac{1}{r}\right\}\|\Phi\|_{L^\infty H^1},\\
|v(t,r)|&\lesssim \min\left\{\frac{1}{r^{3/2}}, \frac{1}{r}\right\}\|v\|_{L^\infty H^1}.
\endaligned
\end{equation*}
Thus, due to \eqref{v-h1} and \eqref{Phi-h1}, we claim \eqref{decay-Phi} and half of \eqref{decay-v}. The other half of \eqref{decay-v} follows as a consequence of \eqref{decay-v-0}.
\end{proof}

Next, we use these bounds to infer asymptotics for $\Phi$ and $\Box\Phi$ in terms of $v$. These are critical in further arguments.
\begin{prop}
Under the assumptions of Theorem \ref{main-th-v-2}, we have
\begin{equation}
|\Phi|\sim |v|+v^2 \quad \text{and} \quad|\Box\Phi|\sim \min\{v^2, |v|^3\}\qquad\text{if}\quad r\ll 1
\label{rll1}
\end{equation}
and
\begin{equation}
\left|\Phi-\frac{\varphi_{\geq 1/2}}{r^3}\right|\sim |v| \quad \text{and} \quad\left|\Box\Phi-\frac{\varphi_{\geq 1/2}}{r^3}\right|\lesssim \frac{|v|}{r^2}\qquad\text{if}\quad r\gtrsim 1.
\label{rgrt1}
\end{equation}
\end{prop}
\begin{proof}
We start by rewriting \eqref{Phi-final} and \eqref{Box-Phi} in the forms
\[
\aligned
\Phi-\frac{\varphi_{\geq 1/2}}{r^3}&=\int_0^v \tilde{A}^{1/2}\,dy,\\
\Box\Phi-\frac{\varphi_{\geq 1/2}}{r^3}&=\int_0^v \left(\tilde{A}^{-1/2}+\tilde{A}^{-3/2}\right)(\tilde{A}-1)\,dy.
\endaligned
\]
If we choose $r<1/2$, then
\[
\tilde{A}=1+\frac{\sin^2(ry)}{r^2}, \qquad \varphi_{\geq 1/2}(r)=0. 
\]
Moreover, by applying \eqref{decay-v}, we can guarantee that $r|v|\leq 1$ if we further calibrate $r$ to be sufficiently small, . Therefore, it follows that
\[
|\Phi|=\int_0^{|v|}\left(1+\frac{\sin^2(ry)}{r^2}\right)^{1/2}dy\sim \int_0^{|v|}\left(1+y\right)dy\sim |v|+v^2
\] 
and
\[
\aligned
|\Box\Phi|&=\int_0^{|v|}\left\{\left(1+\frac{\sin^2(ry)}{r^2}\right)^{-1/2}+\left(1+\frac{\sin^2(ry)}{r^2}\right)^{-3/2}\right\}\frac{\sin^2(ry)}{r^2}\,dy\\
&\sim \int_0^{|v|}\frac{y^2}{1+y}\,dy\sim  \min\{v^2, |v|^3\},
\endaligned
\]
which proves \eqref{rll1}.

If $r\gtrsim 1$, one has
\[
1\leq \tilde{A}\leq 1+\frac{2}{r^2}\sim 1
\] 
and the derivation of \eqref{rgrt1} follows along the same lines. \end{proof}

%%%%%%%%%%%%%%%%%%%%%%%%%%%%%%%%%%%%%%%%%%%%%%%%%%%%

\section{$H^2$-type analysis} 

In this section, our goal is to improve upon \eqref{Phi-h1} and show that
\begin{equation}
\|\Phi\|_{L^\infty \dot{H}^2}+\|\Phi_t\|_{L^\infty \dot{H}^1}+\|\Phi_{tt}\|_{L^\infty L^2}\lesssim 1.
\label{H2-norm}
\end{equation}
First, we write a wave equation for $\Phi_t$, which is then investigated by applying Strichartz estimates. As a consequence, we deduce that both $\Phi_t$ and $\Phi_{tt}$ have the desired Sobolev regularity. Combining this information with the main equation satisfied by $\Phi$ (i.e., \eqref{Box-Phi}), we derive that $\Phi\in L^\infty \dot{H}^2$. Following this, the fixed-time decay estimates \eqref{decay-Phi} and \eqref{decay-v} are upgraded.

\vspace{.1in}
\hrule

%%%%%%%%%%%%%%%%%%%%%%%%%%%%%%%%%%%%%%%%%%%%%%%%%%%

\subsection{Argument for the $\dot{H}^1$ and $L^2$ regularities of  $\Phi_{t}$ and $\Phi_{tt}$}

We commence by differentiating with respect to $t$ the equations \eqref{Phi-final} and \eqref{Box-Phi} and thus obtain
\begin{equation}
\Phi_t=\tA^{1/2}(r,v)v_t=\left(1+\frac{\sin^2u}{r^2}\right)^{1/2}v_t
\label{Phit-final}
\end{equation}
and
\begin{equation}
\aligned
\Box_{4+1}\Phi_t=\Phi_t-\tA^{-3/2}(r,v)\,v_t
= \left(\tA(r,v)-1\right)\left(\tA^{-1}(r,v)+\tA^{-2}(r,v)\right)\Phi_t.
\endaligned
\label{Box-Phit}
\end{equation}

In order to move forward, it is clear that we need more qualitative information on $\tA(r,v)$ and, for this purpose, we use \eqref{decay-v} to easily infer
\begin{equation}
|\sin u|\lesssim \min\left\{\frac{1}{r^{1/2}}, r^{1/2}\right\}.
\label{sinu-r}
\end{equation}
This estimate implies
\begin{equation}
\tA(r,v)-1=|\tA(r,v)-1|\lesssim \min\left\{\frac{1}{r^3}, \frac{1}{r}\right\}
\label{decay-A}
\end{equation}
and, subsequently, 
\begin{equation}
0<\tA^{-1}(r,v)+\tA^{-2}(r,v)\sim \tA^{-1}(r,v)\lesssim 1.
\label{ta-2}
\end{equation}
Now, we can proceed to prove that $\Phi_t$ and $\Phi_{tt}$ have $\dot{H}^1$ and $L^2$ regularities, respectively.

\begin{prop}
Under the assumptions of Theorem \ref{main-th-v-2},
\begin{equation}
\|\Phi_t\|_{L^p L^q}+\|\Phi_{t}\|_{L^\infty \dot{H}^{1}}+\|\Phi_{tt}\|_{L^\infty L^2}\lesssim 1
\label{Phit-h1}
\end{equation}
holds true for all pairs $(p,q)$ satisfying 
\begin{equation}
4\leq p\leq \infty, \qquad 2\leq q<\infty, \qquad \frac{4}{p}+\frac{2}{q}\leq 1, \qquad \frac{1}{p}+\frac{4}{q}=1.
\label{pq-h2}
\end{equation}
\label{prop-Phit-h1}
\end{prop}
\begin{proof}
By applying the Strichartz estimates \eqref{Str-gen} to the wave equation \eqref{Box-Phit} for the case when $\sigma=1$ and $(\bar{p}',\bar{q}')=(1,2)$, it follows that
\begin{equation}
\aligned
\|\Phi_t\|&_{L^pL^q(I\times \R^4)}+\|\Phi_t\|_{L^\infty\dot{H}^1(I\times \R^4)}+\|\Phi_{tt}\|_{L^\infty L^2(I\times \R^4)}\\
&\qquad\lesssim \|\Phi_t(a)\|_{\dot{H}^1( \R^4)}+\|\Phi_{tt}(a)\|_{L^2(\R^4)}+\|\Box\Phi_t\|_{L^1L^2(I\times \R^4)}
\endaligned
\label{Phit-h1-str}
\end{equation}
is valid for all intervals $I=[a,b]\subset [0,T)$ and pairs $(p,q)$ satisfying \eqref{pq-h2}. One such pair is $(p,q)=(7,14/3)$. Next, we use \eqref{decay-A} and \eqref{ta-2} to estimate the last term on the right-hand side as
\begin{equation}
\aligned
\|\Box\Phi_t\|_{L^1L^2(I\times \R^4)}&\lesssim  \|\tA(r,v)-1\|_{L^{7/6}L^{7/2}(I\times \R^4)}\|\Phi_t\|_{L^7L^{14/3}(I\times \R^4)}\\
&\lesssim  |I|^{6/7}\|\Phi_t\|_{L^7L^{14/3}(I\times \R^4)}.
\endaligned
\label{Box-Phit-l1l2}
\end{equation}
Consequently, we infer that  
\begin{equation*}
\aligned
\|\Phi_t\|&_{L^7L^{14/3}(I\times \R^4)}+\|\Phi_t\|_{L^\infty\dot{H}^1(I\times \R^4)}+\|\Phi_{tt}\|_{L^\infty L^2(I\times \R^4)}\\
&\qquad\lesssim \|\Phi_t(a)\|_{\dot{H}^1( \R^4)}+\|\Phi_{tt}(a)\|_{L^2(\R^4)}+|I|^{6/7}\|\Phi_t\|_{L^7L^{14/3}(I\times \R^4)}.
\endaligned
\end{equation*}

If we recall our notational conventions, then, for $|I|\sim 1$, yet sufficiently small, we deduce
\[
M(I)\lesssim \|\Phi_t(a)\|_{\dot{H}^{1}( \R^4)}+\|\Phi_{tt}(a)\|_{L^2(\R^4)},
\]
where
\[
M(I):=\|\Phi_t\|_{L^7L^{14/3}(I\times \R^4)}+\|\Phi_t\|_{L^\infty\dot{H}^{1}(I\times \R^4)}+\|\Phi_{tt}\|_{L^\infty L^2(I \times \R^4)}.
\]
Hence, by choosing $T_1$ to be the maximal length of an interval for which the previous bound holds true, we derive that
\[
\aligned
M([(k+1)T_1, (k+2)T_1])&\lesssim \|\Phi_t((k+1)T_1)\|_{\dot{H}^{1}( \R^4)}+\|\Phi_{tt}((k+1)T_1)\|_{L^2(\R^4)}\\&\lesssim M([kT_1, (k+1)T_1])
\endaligned
\]
holds true for as long as 
\[
0\leq kT_1<(k+2)T_1<T,
\] 
with $k$ being a nonnegative integer. Due to the hypothesis of Theorem \ref{main-th-v-2}, the results of the Appendix yield
\[
M([0,T_1])\lesssim \|\Phi_t(0)\|_{\dot{H}^{1}( \R^4)}+\|\Phi_{tt}(0)\|_{L^2(\R^4)}\lesssim 1,
\]
which, jointly with the previous facts, implies
\[\|\Phi_t\|_{L^7L^{14/3}}+\|\Phi_t\|_{L^\infty\dot{H}^1}+\|\Phi_{tt}\|_{L^\infty L^2}\lesssim 1.
\]

If we return now to \eqref{Box-Phit-l1l2}, then we obtain 
\[
\|\Box\Phi_t\|_{L^1L^2}\lesssim 1.
\] 
Coupled to \eqref{Phit-h1-str}, this estimate forces that
\[
\|\Phi_t\|_{L^pL^q}\lesssim 1
\]
also holds true for all pairs $(p,q)\neq (7,14/3)$ satisfying \eqref{pq-h2} and, thus, concludes the argument.
\end{proof}

\begin{remark}
With obvious modifications determined by the different numerology, the above proposition and its proof match exactly the corresponding result in \cite{GGr-17}. There, we worked with the specific pair $(p,q)=(2,5)$. 
\end{remark}

%%%%%%%%%%%%%%%%%%%%%%%%%%%%%%%%%%%%%%%%%%%%%%%%%%%%

\subsection{$\dot{H}^2$ regularity for $\Phi$ and improved decay estimates}

Based on the previous proposition, we can now finish the argument for \eqref{H2-norm} by showing that $\Phi$ has $\dot{H}^2$ regularity. 

\begin{prop}
Under the assumptions of Theorem \ref{main-th-v-2}, it is true that
\begin{equation}
\|\Phi\|_{L^\infty\dot{H}^2}\lesssim 1. 
\label{Phi-lih2}
\end{equation}
\end{prop}
\begin{proof}
Using \eqref{Phit-h1}, we infer that
\begin{equation*}
\aligned
\|\Phi\|_{L^\infty\dot{H}^2}\sim \|\Delta\Phi\|_{L^\infty L^2}&\leq \|\Phi_{tt}\|_{L^\infty L^2}+\|\Box\Phi\|_{L^\infty L^2}\\
&\lesssim 1+\|\varphi_{>1}\Box \Phi\|_{L^\infty L^2}+\|\varphi_{<1}\Box \Phi\|_{L^\infty L^2}.
\endaligned
\end{equation*}
Next, we rely on \eqref{rgrt1} and \eqref{decay-v} to deduce
\begin{equation*}
\|\varphi_{>1}\Box \Phi\|_{L^\infty L^2}\lesssim \left\|\frac{\varphi_{\geq 1/2}}{r^3}\right\|_{L^\infty L^2}+\left\| \varphi_{>1}\frac{|v|}{r^2}\right\|_{L^\infty L^2}\lesssim 1,
\end{equation*}
while the application of \eqref{rll1}, the Sobolev embeddings \eqref{Sob-gen}, and \eqref{Phi-h1} yields
\begin{equation*}
\|\varphi_{<1}\Box \Phi\|_{L^\infty L^2}\lesssim \|\varphi_{<1} \Phi^2\|_{L^\infty L^2}\lesssim \|\Phi\|^2_{L^\infty L^{4}}\lesssim \|\Phi\|^2_{L^\infty H^1}\lesssim 1.
\end{equation*}
The desired estimate \eqref{Phi-lih2} follows as the joint conclusion of these three bounds.
\end{proof}

\begin{remark}
The proof of the same estimate in \cite{GGr-17} is considerably more involved. It requires both a decomposition in the spatial frequency and proving first the intermediate bound
\[
\|\Phi\|_{L^\infty\dot{H}^{3/2}}\lesssim 1.
\] 
\end{remark}

If we invoke the radial Sobolev inequalities \eqref{rad-Sob-1} and \eqref{rad-Sob-2} and the asymptotic equation \eqref{rll1} in the context of the $\dot{H}^2$ regularity for $\Phi$, then we are able to upgrade the previous decay estimates satisfied by $\Phi$, $v$, and $\tA(r,v)-1$. 

\begin{prop}
Under the assumptions of Theorem \ref{main-th-v-2}, we have that 
\begin{align}
|\Phi(t,r)|&\lesssim \min\left\{\frac{1}{r^{3/2}}, \frac{1}{r^{\epsilon}}\right\},\label{decay-Phi-2}\\
|v(t,r)|&\lesssim \min\left\{\frac{1}{r^{3/2}}, \frac{1}{r^{\epsilon/2}}\right\},\label{decay-v-2}\\
|\tA(r,v)-1|&\lesssim \min\left\{\frac{1}{r^3}, \frac{1}{r^{\epsilon}}\right\},\label{decay-A-2}
\end{align}
are valid for a fixed, yet arbitrary, $0<\epsilon<3/2$.
\end{prop}

\begin{remark}
When compared to the similar result in \cite{GGr-17}, the less precise nature of this proposition is motivated by the radial Sobolev inequality \eqref{rad-Sob-1} being limited in applicability to Sobolev regularities in the range $1/2<s<2$.
\end{remark}

%%%%%%%%%%%%%%%%%%%%%%%%%%%%%%%%%%%%%%%%%%%%%%%%%%%%

\section{$H^3$-type analysis} 

Here, we take the next step in improving the Sobolev regularities for $\Phi$ and its derivatives by arguing that
\begin{equation}
\|\Phi\|_{L^\infty \dot{H}^3}+\|\Phi_t\|_{L^\infty \dot{H}^2}+\|\Phi_{tt}\|_{L^\infty \dot{H}^1}+\|\Phi_{ttt}\|_{L^\infty L^2}\lesssim 1.
\label{H3-norm}
\end{equation}
We proceed in a similar fashion to the last section and begin by writing a wave equation for $\Phi_{tt}$, which is analyzed through Strichartz estimates. This provides us with the desired regularity for  both $\Phi_{tt}$ and $\Phi_{ttt}$. Following this, we are able to deduce that $\Phi_t\in L^\infty \dot{H}^2$ and $\Phi\in L^\infty \dot{H}^3$ by investigating equations satisfied by $\Phi_t$ and $\Phi_r$, respectively. As a consequence of \eqref{H3-norm}, we can further upgrade the decay rates for  $\Phi$, $v$, and $\tA(r,v)-1$.

\vspace{.1in}
\hrule

\subsection{Derivation of $\dot{H}^1$ and $L^2$ regularities for  $\Phi_{tt}$ and $\Phi_{ttt}$}

If we differentiate \eqref{Phit-final} and \eqref{Box-Phit} with respect to $t$, we obtain
\begin{equation}
\Phi_{tt}=\tA^{1/2}(r,v)v_{tt}+ \tA^{-1/2}(r,v)\,\frac{\sin(2u)}{2r}\,v^2_{t}
\label{Phitt-final}
\end{equation}
and
\begin{equation}
\aligned
\Box_{4+1}\Phi_{tt}=& \left(\tA(r,v)-1\right)\left(\tA^{-1}(r,v)+\tA^{-2}(r,v)\right)\Phi_{tt}\\
&+2\tA^{-3}(r,v)\partial_t(\tA(r,v))\Phi_{t}.
\endaligned
\label{Box-Phitt}
\end{equation}
An important remark is that \eqref{Phit-final} and \eqref{decay-A} imply
\begin{equation}
|\partial_t(\tA(r,v))|=\frac{|\sin(2u)|}{r}\,|v_t|\lesssim \tA^{1/2}(r,v)\,|v_t|=|\Phi_t|
\label{tA-Phit}
\end{equation}
and 
\[
0<\tA^{-3}(r,v)\lesssim 1,
\]
which, together with \eqref{ta-2}, yield 
\begin{equation}
|\Box\Phi_{tt}|\lesssim\Phi_{t}^2+ |(\tA(r,v)-1)\Phi_{tt}|.
\label{bd-box-phitt}
\end{equation}
This is all that is needed to derive the desired regularities for $\Phi_{tt}$ and $\Phi_{ttt}$.

\begin{prop}
Under the assumptions of Theorem \ref{main-th-v-2},
\begin{equation}
\|\Phi_{tt}\|_{L^p L^q}+\|\Phi_{tt}\|_{L^\infty \dot{H}^{1}}+\|\Phi_{ttt}\|_{L^\infty L^2}\lesssim 1
\label{Phitt-h1}
\end{equation}
holds true for all pairs $(p,q)$ satisfying \eqref{pq-h2}.
\label{prop-Phitt-h1}
\end{prop}
\begin{proof}
The argument follows in the footsteps of the one for Proposition \ref{prop-Phit-h1}, in the sense that we start by applying the Strichartz estimates \eqref{Str-gen} to the equation \eqref{Box-Phitt}, i.e., 
\begin{equation*}
\aligned
\|\Phi_{tt}\|&_{L^pL^q(I\times \R^4)}+\|\Phi_{tt}\|_{L^\infty\dot{H}^1(I\times \R^4)}+\|\Phi_{ttt}\|_{L^\infty L^2(I\times \R^4)}\\
&\qquad\lesssim \|\Phi_{tt}(a)\|_{\dot{H}^1( \R^4)}+\|\Phi_{ttt}(a)\|_{L^2(\R^4)}+\|\Box\Phi_{tt}\|_{L^1L^2(I\times \R^4)},
\endaligned
\end{equation*}
which are valid under the same restrictions for which \eqref{Phit-h1-str} holds true. Next, we use \eqref{bd-box-phitt}, the Sobolev embeddings \eqref{Sob-gen}, \eqref{Phit-h1}, and \eqref{decay-A} to infer that
\begin{equation*}
\aligned
\|\Box\Phi_{tt}\|_{L^1L^2(I\times \R^4)}
\lesssim\, &\|\Phi_t\|^2_{L^2L^{4}(I\times \R^4)}+ \|\tA(r,v)-1\|_{L^{7/6}L^{7/2}(I\times \R^4)}\|\Phi_{tt}\|_{L^7L^{14/3}(I\times \R^4)}\\
\lesssim\, &|I|\|\Phi_t\|^2_{L^\infty H^{1}(I\times \R^4)}+ |I|^{6/7}\|\Phi_{tt}\|_{L^7L^{14/3}(I\times \R^4)}\\
\lesssim\, &1+ |I|^{6/7}\|\Phi_{tt}\|_{L^7L^{14/3}(I\times \R^4)}.
\endaligned
\end{equation*}
Therefore, we claim that 
\[
\aligned
\|\Phi_{tt}\|_{L^7L^{14/3}(I\times \R^4)}+\|\Phi_{tt}\|_{L^\infty\dot{H}^{1}(I\times \R^4)}&+\|\Phi_{ttt}\|_{L^\infty L^2(I \times \R^4)}\\
&\lesssim 1+ \|\Phi_{tt}(a)\|_{\dot{H}^{1}( \R^4)}+\|\Phi_{ttt}(a)\|_{L^2(\R^4)},
\endaligned
\]
for $|I|\sim 1$, yet small enough. As before, we invoke the Appendix to deduce 
\[
\|\Phi_{tt}\|_{L^7L^{14/3}}+\|\Phi_{tt}\|_{L^\infty\dot{H}^1}+\|\Phi_{ttt}\|_{L^\infty L^2}\lesssim 1,
\]
which is enough to argue that
\[
\|\Phi_{tt}\|_{L^pL^q}\lesssim 1
\]
holds true for all pairs $(p,q)\neq (7,14/3)$ satisfying \eqref{pq-h2}.
\end{proof}

\subsection{$\dot{H}^3$ and $\dot{H}^2$ regularities for  $\Phi$ and $\Phi_{t}$ and further improvement of the decay information}

As a direct consequence of the previous proposition, we obtain the $\dot{H}^2$ regularity for  $\Phi_{t}$.

\begin{prop}
Under the assumptions of Theorem \ref{main-th-v-2}, it is true that
\begin{equation}
\|\Phi_t\|_{L^\infty\dot{H}^2}\lesssim 1. 
\label{Phit-lih2}
\end{equation}
\end{prop}
\begin{proof}
By virtue of \eqref{Box-Phit}, \eqref{ta-2}, \eqref{Phitt-h1}, \eqref{decay-A-2}, and \eqref{Phit-h1}, we have  that
\[
\aligned
\|\Phi_t\|_{L^\infty\dot{H}^2}\sim \|\Delta\Phi_t\|_{L^\infty L^2}&\lesssim \|\Phi_{ttt}\|_{L^\infty L^2}+\|\Box\Phi_t\|_{L^\infty L^2}\\
&\lesssim 1+\|\tA(r,v)-1\|_{L^\infty L^{4}}\|\Phi_t\|_{L^\infty L^{4}}\\
&\lesssim 1.
\endaligned
\] 
\end{proof}

A more intricate argument is needed to derive the corresponding Sobolev regularity for  $\Phi$.

\begin{prop}
Under the assumptions of Theorem \ref{main-th-v-2}, it is true that
\begin{equation}
\|\Phi\|_{L^\infty\dot{H}^3}\lesssim 1. 
\label{Phi-lih3}
\end{equation}
\end{prop}
\begin{proof}
We start by relying on \eqref{Phitt-h1} to deduce
\begin{equation}
\aligned
\|\Phi\|_{L^\infty\dot{H}^3}\sim \|D\Delta\Phi\|_{L^\infty L^2}&\lesssim \|D\Phi_{tt}\|_{L^\infty L^2}+\|D\Box\Phi\|_{L^\infty L^2}\\
&\sim \|\Phi_{tt}\|_{L^\infty \dot{H}^1}+\|\partial_r\Box\Phi\|_{L^\infty L^2}\\
&\lesssim 1+\|\partial_r\Box\Phi\|_{L^\infty L^2}.
\endaligned
\label{dr-box-phi-l2}
\end{equation}
If we differentiate \eqref{Box-Phi} with respect to $r$, we obtain
\begin{equation*}
\aligned
\partial_r\Box\Phi=\,&\frac{1}{2}\int_0^v \left\{\left(\tA^{-1/2}+3\tA^{-5/2}\right)\tA_r\right\}dy\\&+\left(\tA(r,v)-1\right)\left(\tA^{-1}(r,v)+\tA^{-2}(r,v)\right)\tA^{1/2}(r,v)v_r+\frac{\varphi_{\geq 1/2}}{r^3}
\endaligned
\end{equation*}
and, taking into account \eqref{ta-formula}, we infer that
\begin{equation}
\left|\partial_r\Box\Phi\right|\lesssim \int_0^{|v|} \left\{|\tA_r|\right\}dy+|\tA(r,v)-1|\tA^{1/2}(r,v)|v_r|+\frac{|\varphi_{\geq 1/2}|}{r^3}.
\label{dr-box-phi}
\end{equation}
One can easily verify that
\begin{equation}
\left\|\frac{\varphi_{\geq 1/2}}{r^3}\right\|_{L^\infty L^2}\lesssim 1
\label{phi-g12}
\end{equation}
and, consequently, we focus on the other two terms on the right-hand side of the previous estimate.

In what concerns the integral term, the joint application of \eqref{ta-formula}, \eqref{ta-r}, \eqref{ta-r-g1}, \eqref{ta-r-l1}, and \eqref{decay-v-2} yields
\begin{equation}
\int_0^{|v|}\left\{|\tA_r|\right\}dy\,\lesssim \,\frac{|v|+v^2}{r^2}\lesssim \frac{1}{r^{7/2}}
\label{int-tar-g1}
\end{equation}
when $r\geq 1$, and
\begin{equation}
\int_0^{|v|}\left\{|\tA_r|\right\}dy\,\lesssim \,r|v|^5\lesssim r^{1-5\epsilon/2}
\label{int-tar-l1}
\end{equation}
when $r<1$. Thus, we derive immediately that
\begin{equation}
\left\|\int_0^{|v|}\left\{|\tA_r|\right\}dy\right\|_{L^\infty L^2}\lesssim 1.
\label{int-tar}
\end{equation}

Finally, for the term in \eqref{dr-box-phi} having $v_r$ as a factor, we work with \eqref{Phi-final} to deduce
\begin{equation}
\Phi_r=\tA^{1/2}(r,v)v_r+\frac{1}{2}\int_0^v\left\{\tA^{-1/2}\tA_r\right\}dy+\frac{\varphi_{\geq 1/2}}{r^3}.
\label{phir-vr}
\end{equation}
It is an easy verification that
\[
\left\|\frac{\varphi_{\geq 1/2}}{r^3}\right\|_{L^\infty L^{4}}\lesssim 1
\]
and, using \eqref{int-tar-g1} and \eqref{int-tar-l1}, we obtain
\[
\left\|\int_0^{v}\left\{\tA^{-1/2}\tA_r\right\}dy\right\|_{L^\infty L^{4}}\lesssim 1.
\]
Moreover, on the basis of \eqref{Sob-gen} and \eqref{Phi-lih2}, we infer that
\[
\|\Phi_r\|_{L^\infty L^{4}}\lesssim \|\Phi_r\|_{L^\infty H^{1}}\lesssim 1. 
\]
Hence, by putting together the last four mathematical statements, we derive that 
\[
\|\tA^{1/2}(r,v)v_r\|_{L^\infty L^{4}}\lesssim 1.
\]
If we combine this estimate with \eqref{decay-A-2}, then we arrive at
\[
\|(\tA(r,v)-1)\tA^{1/2}(r,v)v_r\|_{L^\infty L^2}\lesssim \|\tA(r,v)-1\|_{L^\infty L^4}\|\tA^{1/2}(r,v)v_r\|_{L^\infty L^{4}}\lesssim 1,
\]
which, jointly with \eqref{dr-box-phi-l2}, \eqref{dr-box-phi}, \eqref{phi-g12}, and \eqref{int-tar}, implies \eqref{Phi-lih3}.
\end{proof}

Consequently, we can argue as in the last section and further upgrade the decay bounds for $\Phi$, $v$, and $\tA(r,v)-1$.

\begin{prop}
Under the assumptions of Theorem \ref{main-th-v-2}, we have
\begin{align}
|\Phi(t,r)|&\lesssim \frac{1}{1+r^{3/2}},\label{decay-Phi-3}\\
|v(t,r)|&\lesssim \frac{1}{1+r^{3/2}},\label{decay-v-3}\\
|\tA(r,v)-1|&\lesssim \frac{1}{1+r^3}.\label{decay-A-3}
\end{align}
\end{prop}

\begin{remark}
It is easy to see that \eqref{decay-v-3} implies what is required of $v$ in the main estimate to be proved (i.e., \eqref{livr}):
\begin{equation}
\|(1+r)|v|\|_{L^\infty_{t,x}}\leq  \|(1+r^{3/2})|v|\|_{L^\infty_{t,x}}\lesssim 1.
\label{li-v}
\end{equation}
\end{remark}

\begin{remark}
For the portion of the same inequality involving $\nabla v$, what we have so far yields
\begin{equation}
\|(1+r)|\nabla v|\|_{L^\infty_{t,x}([0,T)\times \{r\geq 1\})}\lesssim 1.
\label{li-nabla-v-1}
\end{equation}
Indeed, we derive from \eqref{Phit-final} and \eqref{phir-vr} that 
\begin{equation}
v_t\,=\,\tA^{-1/2}(r,v)\Phi_t
\label{vt-phit}
\end{equation}
and
\begin{equation}
\aligned
v_r\,=\,\tA^{-1/2}(r,v)\left(\Phi_r-\frac{1}{2}\int_{0}^{v}\left\{\tA^{-1/2}\tA_r\right\}dy-\frac{\varphi_{\geq 1/2}}{r^3}\right),
\endaligned
\label{vr-phir}
\end{equation}
respectively. Hence, in the regime when $r\geq 1$, with the help of \eqref{ta-r}, \eqref{ta-r-g1}, \eqref{decay-A}, and \eqref{decay-v}, we deduce
\[
\aligned
r(|v_r|+|v_t|)&\lesssim r\left(|\Phi_r|+|\Phi_t|+\int_0^{|v|}\left\{\frac{1+|y|}{r^2}\right\}\,dy+\frac{\left|\varphi_{\geq 1/2}\right|}{r^3}\right)\\
&\lesssim r\left(|\Phi_r|+|\Phi_t|+\frac{|v|+v^2}{r^2}+\frac{1}{r^3}\right)\\
&\lesssim r\left(|\Phi_r|+|\Phi_t|+\frac{1}{r^3}\right).
\endaligned
\]
In the end, if we rely on \eqref{rad-Sob-1} and \eqref{H2-norm}, we obtain
\[
r(|v_r|+|v_t|)\lesssim \|\Phi_r\|_{L^\infty\dot{H}^{1}} +  \|\Phi_t\|_{L^\infty\dot{H}^{1}} +1\lesssim 1,
\]
which proves \eqref{li-nabla-v-1}.
\end{remark}

\begin{remark}
The results of this section align themselves perfectly, both in statement and approach, with the corresponding ones in \cite{GGr-17}.
\end{remark}

%%%%%%%%%%%%%%%%%%%%%%%%%%%%%%%%%%%%%%%%%%%%%%%%%%%

\section{Final estimates and conclusion of the argument}

By taking advantage of the estimates \eqref{li-v} and \eqref{li-nabla-v-1}, it follows that the argument for \eqref{livr} (and thus the proof of Theorem \ref{main-th-v-2}) is concluded, if we show that   
\begin{equation}
\|\nabla v\|_{L^\infty_{t,x}([0,T)\times \{r< 1\})}\lesssim 1
\label{li-nabla-v-2}
\end{equation}
holds true. However, when $r<1$, we can use \eqref{vr-phir}, \eqref{vt-phit}, \eqref{ta-r}, \eqref{ta-r-l1}, \eqref{decay-A}, and \eqref{decay-v-3} to infer that
\[
\aligned
|v_r|+|v_t|&\lesssim |\Phi_r|+|\Phi_t|+\int_0^{|v|}\left\{ry^4\right\}\,dy+\frac{\left|\varphi_{\geq 1/2}\right|}{r^3}\\
&\lesssim |\Phi_r|+|\Phi_t|+r|v|^5+1\\
&\lesssim |\Phi_r|+|\Phi_t|+1,
\endaligned
\]
which leads to
\begin{equation*}
\|\nabla v\|_{L^\infty_{t,x}([0,T)\times \{r< 1\})}\lesssim \|\Phi_r\|_{L^\infty_{t,x}} +  \|\Phi_t\|_{L^\infty_{t,x}} +1.
\end{equation*}
Next, we apply \eqref{Phi-h1} and the classical Sobolev embedding \eqref{Sob-classic} to see that we are done once we prove that
\begin{equation}
\|\Phi\|_{L^\infty \dot{H}^{s}}+\|\Phi_t\|_{L^\infty \dot{H}^{s-1}}\lesssim 1
\label{phi-phit-hs}
\end{equation}
is valid.

The approach we take in arguing for this bound is to rely first on energy estimates applied to \eqref{Box-Phit}
in order to derive
\begin{equation*}
\|\Phi_t\|_{L^\infty \dot{H}^{s-1}}+\|\Phi_{tt}\|_{L^\infty \dot{H}^{s-2}}\lesssim 1.
\end{equation*}
If we couple this information with the original equation \eqref{Box-Phi} satisfied by $\Phi$, then we also deduce 
\begin{equation*}
\|\Phi\|_{L^\infty \dot{H}^{s}}\lesssim 1
\end{equation*}
and the proof of \eqref{phi-phit-hs} is finished. 

\vspace{.1in}
\hrule

%%%%%%%%%%%%%%%%%%%%%%%%%%%%%%%%%%%%%%%%%%%%%%%%%%%

\subsection{New qualitative bounds for $v$ and $\tA(r,v)$}

It is evident from the strategy outlined above that what we need to have for the subsequent analysis is 
more qualitative information on $v$ and $\tA(r,v)$, in addition to  \eqref{v-h1}, \eqref{decay-v-3}, and \eqref{decay-A-3}. For this purpose, we begin by proving the following result.

\begin{prop}
Under the assumptions of Theorem \ref{main-th-v-2}, we have
\begin{align}
\|\nabla v\|_{L^\infty L^4}+&\|\Delta v\|_{L^\infty L^2}\lesssim 1,\label{nabla-delta-v}\\
\|\nabla (\tA(r,v))\|_{L^\infty L^4}+&\|\Delta( \tA(r,v))\|_{L^\infty L^2}\lesssim 1\label{nabla-delta-A}.
\end{align}
\end{prop}

\begin{proof}
First, we show that both $L^\infty L^4$ norms are finite. By applying \eqref{tA-Phit}, the Sobolev embeddings \eqref{Sob-gen}, \eqref{phitlil2}, and \eqref{Phit-h1}, we infer that
\begin{equation}
\| \partial_t(\tA(r,v))\|_{L^\infty L^4}\lesssim \|\Phi_t\|_{L^\infty L^4}\lesssim \|\Phi_t\|_{L^\infty H^1}\lesssim 1,
\label{tA-lil4}
\end{equation}
which, together with \eqref{Phit-final}, implies
\begin{equation}
\| v_t\|_{L^\infty L^4}\lesssim \|\Phi_t\|_{L^\infty L^4}\lesssim 1.
\label{vt-lil4}
\end{equation}

Next, if we rely on \eqref{vr-phir} combined with \eqref{ta-r}, \eqref{ta-r-g1}, and \eqref{ta-r-l1}, then we obtain
\begin{equation*}
|v_r|\lesssim |\Phi_r|+\frac{\left|\varphi_{\geq 1/2}\right|}{r^3}+
\begin{cases}
\frac{|v|+v^2}{r^2}, \, &r\geq 1,\\
r|v|^5, \, &r< 1.
\end{cases}
\end{equation*}
Therefore, by also factoring in the Sobolev embeddings \eqref{Sob-gen}, \eqref{decay-v-3}, \eqref{philil2}, and \eqref{Phi-lih2}, it follows that
\begin{equation}
\aligned
\| v_r\|_{L^\infty L^4}\lesssim &\,\|\Phi_r\|_{L^\infty L^4}+\left\|\frac{1}{r^3}\right\|_{L^\infty L^4([0,T)\times \{r\geq 1/2\})}\\&+ \left\|\frac{1}{r^{7/2}}\right\|_{L^\infty L^4([0,T)\times \{r\geq 1\})}+ \left\|r\right\|_{L^\infty L^4([0,T)\times \{r< 1\})}\\\lesssim &\, \|\Phi\|_{L^\infty H^{2}}+1\\ \lesssim &\, 1.
\endaligned
\label{vr-lil4}
\end{equation}
We conclude this part of the argument by using the formula \eqref{ta-formula}, \eqref{ta-r}, \eqref{ta-r-g1}, \eqref{ta-r-l1}, \eqref{decay-v-3}, and the previous estimate to deduce
\begin{equation*}
\aligned
\| \partial_r&(\tA(r,v))\|_{L^\infty L^4}\\&\lesssim \left\|\frac{1+|v|}{r^2}+\frac{|v_r|}{r}\right\|_{L^\infty L^4([0,T)\times \{r\geq 1\})}+ \left\|rv^4+|v||v_r|\right\|_{L^\infty L^4([0,T)\times \{r< 1\})}\\&\lesssim \left\|\frac{1}{r^2}\right\|_{L^\infty L^4([0,T)\times \{r\geq 1\})}+ \left\|r\right\|_{L^\infty L^4([0,T)\times \{r< 1\})}+\left\|v_r\right\|_{L^\infty L^4}\\ &\lesssim 1.
\endaligned
\end{equation*}

Next, we prove the finiteness of the $L^\infty L^2$ norm in \eqref{nabla-delta-v} by showing that
\begin{equation}
\|v_{tt}\|_{L^\infty L^2}+\|\Box v\|_{L^\infty L^2}\lesssim 1.
\label{vtt-dv}
\end{equation}
We derive directly from \eqref{Phitt-final} that
\[
|v_{tt}|\lesssim |\Phi_{tt}|+v_t^2
\]
and, consequently, we infer due to \eqref{Phit-h1} and \eqref{vt-lil4} that
\[
\|v_{tt}\|_{L^\infty L^2}\lesssim \|\Phi_{tt}\|_{L^\infty L^2}+\|v_t\|_{L^\infty L^4}^2\lesssim 1.
\]

To estimate the $L^\infty L^2$ norm of $\Box v$, we look at the equation \eqref{main-v} and analyze individually each term on its right-hand side. First, if we rely on the definitions of $\varphi$ and $\varphi_{>1}$ and \eqref{v-l2}, then we easily obtain
\begin{equation*}
\left\| \frac{1}{r}\Delta_2 \varphi\right\|_{L^\infty L^2}+\left\|\frac{1}{r^2}\varphi_{>1} v\right\|_{L^\infty L^2} \lesssim \left\|\frac{1}{r}\right\|_{L^\infty L^2([0,T)\times \{1\leq r\leq 2\})}+\|v\|_{L^\infty L^2}\lesssim 1.
\end{equation*}
Next, by applying \eqref{main-2}, \eqref{utov}, and \eqref{decay-v-3}, we deduce
\begin{equation*}
\aligned
\frac{1}{r}\left|\varphi_{>1}N(r, rv+\varphi, \nabla(rv+\varphi))\right|&\lesssim \frac{1}{r}|\varphi_{>1}| \left(\frac{1+|\nabla(rv+\varphi)|^2}{r^2}+\frac{|v+rv_r+\varphi_r|}{r^3}\right)\\
&\lesssim |\varphi_{>1}| \left(\frac{1}{r^3}+\frac{|\nabla v|^2}{r}+\frac{v^2}{r^3}\right).
\endaligned
\end{equation*}
Therefore, with the help of \eqref{v-l2}, \eqref{decay-v-3}, \eqref{vt-lil4}, and \eqref{vr-lil4}, we derive
\begin{equation*}
\aligned
&\left\|\frac{1}{r}\varphi_{>1}N(r, rv+\varphi, \nabla(rv+\varphi))\right\|_{L^\infty L^2}\\
&\qquad\qquad\lesssim \left\|\frac{1}{r^3}\right\|_{L^\infty L^2([0,T)\times \{ r\geq 1/2\})}+\|\nabla v\|^2_{L^\infty L^4}+\|v\|_{L^\infty_{t,x}}\|v\|_{L^\infty L^2}\\&\qquad\qquad\lesssim 1.
\endaligned
\end{equation*}
Finally, based on \eqref{Nv-l1}, we infer that 
\begin{equation*}
\left|\varphi_{<1}\left(\frac{1}{r}N(r, rv, \nabla(rv))+\frac{1}{r^2}v\right)\right|\lesssim |\varphi_{<1}|\left(|v|^3+|v|^5+|v||\nabla v|^2+rv^4|v_r|\right)
\end{equation*}
and, consequently, 
\begin{equation*}
\aligned
&\left\|\varphi_{<1}\left(\frac{1}{r}N(r, rv, \nabla(rv))+\frac{1}{r^2}v\right)\right\|_{L^\infty L^2}\\
&\qquad \qquad\qquad\lesssim \left(\|v\|_{L^\infty_{t,x}}^2+\|v\|_{L^\infty_{t,x}}^4 \right)\|v\|_{L^\infty L^2}+\|v\|_{L^\infty_{t,x}}\|\nabla v\|^2_{L^\infty L^4}\\&\qquad \qquad\qquad\quad+\|v\|_{L^\infty_{t,x}}^{7/2}\|v\|^{1/2}_{L^\infty L^2}\|\nabla v\|_{L^\infty L^4}.
\endaligned
\end{equation*}
The desired estimate is then obtained as a result of \eqref{decay-v-3}, \eqref{v-l2}, \eqref{vt-lil4}, and \eqref{vr-lil4}. Hence, the argument for the finiteness of $\|\Box v\|_{L^\infty L^2}$ is concluded and we have also finished the proof of \eqref{nabla-delta-v}.

Following this, a straightforward calculation based on \eqref{ta-formula} yields
\begin{equation}
\aligned
\Delta (\tA(r,v))=&-\frac{\sin(2u)}{r^3}\left(v+rv_r+\varphi_r\right)+2\,\frac{\cos(2u)}{r^2}\left(v+rv_r+\varphi_r\right)^2\\
&+\frac{\sin(2u)}{r^2}\left(r\Delta v-v_r+\varphi_{rr}\right).
\endaligned
\label{delta-ta}
\end{equation}
If we rely on the definition of $\varphi$ and \eqref{nabla-delta-v}, then we obtain
\begin{equation}
\aligned
\left\|\frac{\cos(2u)}{r^2}\left(rv_r+\varphi_r\right)^2\right\|_{L^\infty L^2}&\lesssim \|v_r\|^2_{L^\infty L^4}+\left\|\frac{\varphi_r}{r}\right\|^2_{L^\infty L^4}\\&\lesssim \|\nabla_x v\|^2_{L^\infty L^4}+\left\|\frac{1}{r}\right\|^2_{L^\infty L^4([0,T)\times \{1\leq r\leq 2\})}\\&\lesssim 1.
\endaligned
\label{delta-ta-1}
\end{equation}
Next, the elementary bound
\[
\frac{|\sin(2u)|}{r}\lesssim  \tA^{1/2}(r,v)
\] 
implies
\[
\left|\frac{\sin(2u)}{r^3}\left(rv_r+\varphi_{r}\right)\right|\lesssim  \tA^{1/2}(r,v)\left(\frac{|\nabla_xv|}{r}+\frac{|\varphi_{r}|}{r^2}\right)
\]
and
\[
\left|\frac{\sin(2u)}{r^2}\left(r\Delta v-v_r+\varphi_{rr}\right)\right|\lesssim  \tA^{1/2}(r,v)\left(|\Delta v|+\frac{|\nabla_xv|+|\varphi_{rr}|}{r}\right).
\]
Thus, we can apply \eqref{Hardy}, the definition of $\varphi$, \eqref{decay-A-3}, and \eqref{nabla-delta-v} to deduce
\begin{equation}
\aligned
&\left\|\frac{\sin(2u)}{r^3}\left(rv_r+\varphi_{r}\right)\right\|_{L^\infty L^2}\\
&\qquad\quad\lesssim \left\| \tA^{1/2}(r,v)\right\|_{L^\infty_{t,x}}\left(\left\|\frac{\nabla_xv}{r}\right\|_{L^\infty L^2}+\left\|\frac{\varphi_{r}}{r^2}\right\|_{L^\infty L^2}\right)\\&\qquad\quad\lesssim  \left\| \tA^{1/2}(r,v)\right\|_{L^\infty_{t,x}}\left(\left\|\Delta v\right\|_{L^\infty L^2}+\left\|\frac{1}{r^2}\right\|_{L^\infty L^2([0,T)\times \{1\leq r\leq 2\})}\right)\\&\qquad\quad\lesssim 1
\endaligned
\label{delta-ta-21}
\end{equation}
and
\begin{equation}
\aligned
&\left\|\frac{\sin(2u)}{r^2}\left(r\Delta v-v_r+\varphi_{rr}\right)\right\|_{L^\infty L^2}\\
&\qquad\quad\lesssim \left\| \tA^{1/2}(r,v)\right\|_{L^\infty_{t,x}}\left(\left\|\Delta v\right\|_{L^\infty L^2}+\left\|\frac{\nabla_xv}{r}\right\|_{L^\infty L^2}+\left\|\frac{\varphi_{rr}}{r}\right\|_{L^\infty L^2}\right)\\&\qquad\quad\lesssim  \left\| \tA^{1/2}(r,v)\right\|_{L^\infty_{t,x}}\left(\left\|\Delta v\right\|_{L^\infty L^2}+\left\|\frac{1}{r}\right\|_{L^\infty L^2([0,T)\times \{1\leq r\leq 2\})}\right)\\&\qquad\quad\lesssim 1.
\endaligned
\label{delta-ta-2}
\end{equation}
If we argue as we did for \eqref{ta-r-g1}-\eqref{ta-r-l1}, then we derive
\begin{equation}
\left|-\frac{\sin(2u)}{r^3}v+2\,\frac{\cos(2u)}{r^2}v^2\right|\lesssim
\begin{cases}
\frac{|v|}{r^3}+\frac{v^2}{r^2}, \, &r\geq 1,\\
v^4, \, &r< 1.
\end{cases}
\label{sin2u-cos2u}
\end{equation}
Thus, on the basis of \eqref{v-l2} and \eqref{decay-v-3}, we infer that
\begin{equation}
\left\|-\frac{\sin(2u)}{r^3}v+2\,\frac{\cos(2u)}{r^2}v^2\right\|_{L^\infty L^2}\lesssim (1+\|v\|^3_{L^\infty_{t,x}})\|v\|_{L^\infty L^2}\lesssim 1.
\label{sin2u-cos2u-norm}
\end{equation}
Together with \eqref{delta-ta}-\eqref{delta-ta-2}, this estimate implies 
\begin{equation}
\|\Delta (\tA(r,v))\|_{L^\infty L^2}\lesssim 1,
\label{delta-a-lil2}
\end{equation}
which ends the argument for \eqref{nabla-delta-A} and the whole proof of this proposition.
\end{proof}

\begin{remark}
This result has the statement and most of the argument in common with its counterpart in \cite{GGr-17}. However, in proving \eqref{delta-a-lil2} for the Skyrme model, one uses \eqref{Hardy} to obtain
\begin{equation*}
\left\|\frac{\cos(2u)}{r^2}v^2\right\|_{L^\infty L^2([0,T)\times \R^5)}\lesssim\left\|\frac{v}{r}\right\|^2_{L^\infty L^4([0,T)\times \R^5)}\lesssim\left\|\nabla_x v\right\|^2_{L^\infty L^4([0,T)\times \R^5)}\lesssim 1.
\end{equation*}
Here, on the other hand, we can't apply \eqref{Hardy} because we are in the case when $p=n=4$. This is why it is crucial that we have the pairing described by \eqref{sin2u-cos2u}, which leads to \eqref{sin2u-cos2u-norm}.
\end{remark}

Following this, given that both $\tA^{-1}$ and $\tA^{-2}$ are present on the right-hand side of \eqref{Box-Phit}, we also need to have estimates for derivatives of $\tA^{-1}$. For this purpose, we rely on the subsequent proposition, which matches perfectly the corresponding one in \cite{GGr-17}. Moreover, the two arguments are identical, with the obvious modification that changes the domain of the functional spaces from $\R^5$ to $\R^4$. This is why we just state the result here and refer the reader to \cite{GGr-17} for details on its proof.  

\begin{prop}
Under the assumptions of Theorem \ref{main-th-v-2}, the fixed-time bound
\begin{equation}
\|D^\sigma(\tA^{-1}(r,v))\|_{L^p(\R^4)}\lesssim \|D^\sigma(\tA(r,v))\|_{L^p(\R^4)}
\label{ta-inv-ta}
\end{equation}
holds true uniformly on $[0,T)$ for all $1<\sigma<2$ and $1<p<\infty$. 
\end{prop}
%%%%%%%%%%%%%%%%%%%%%%%%%%%%%%%%%%%%%%%%%%%%%%%%%%%

\subsection{Improved Sobolev regularities for $\Phi_t$ and $\Phi$}
We can now proceed to upgrade the $H^2$ and $H^3$ regularities for  $\Phi_t$ and $\Phi$, respectively, to the level of the ones featured in \eqref{phi-phit-hs}. As described in the start of this section, we first focus on $\Phi_t$.

\begin{prop}
Under the assumptions of Theorem \ref{main-th-v-2}, with $s>3$ replaced by $3<s<4$,
\begin{equation}
\|\Phi_t\|_{L^\infty \dot{H}^{s-1}}+\|\Phi_{tt}\|_{L^\infty \dot{H}^{s-2}}\lesssim 1
\label{phit-phitt-hs}
\end{equation}
is valid.
\label{prop-phit}
\end{prop}

\begin{proof}
We begin by invoking the energy-type estimate \eqref{en-hs} in the context of \eqref{Box-Phit} to deduce
\begin{equation}
\aligned
\|\Phi_t\|_{L^\infty \dot{H}^{s-1}}+\|\Phi_{tt}\|_{L^\infty \dot{H}^{s-2}}\lesssim\, &\|\Phi_t(0)\|_{\dot{H}^{s-1}( \R^4)}+\|\Phi_{tt}(0)\|_{\dot{H}^{s-2}(\R^4)}\\&+\|\Box\Phi_t\|_{L^1\dot{H}^{s-2}}.
\endaligned
\end{equation}
The Appendix claims that
\begin{equation*}
\|\Phi_t(0)\|_{\dot{H}^{s-1}( \R^4)}+\|\Phi_{tt}(0)\|_{\dot{H}^{s-2}(\R^4)}\lesssim 1
\end{equation*}
and, thus, for deriving \eqref{phit-phitt-hs}, it is enough to argue that
\begin{equation}
\|\Box\Phi_t\|_{L^1\dot{H}^{s-2}}\lesssim 1.
\label{box-phit-hs-1}
\end{equation}
Based on \eqref{Box-Phit} and the fractional Leibniz bound \eqref{Lbnz-0}, we infer that
\begin{equation}
\aligned
\|\Box\Phi_t\|&_{L^1\dot{H}^{s-2}}\\
\lesssim\,&\|D^{s-2}(\tA(r,v))\|_{L^\infty L^{4/(s-2)}}\|\Phi_t\|_{L^\infty L^{4/(4-s)}}\\
&\cdot\left(\|\tA^{-1}(r,v)\|_{L^\infty_{t,x}}+\|\tA^{-1}(r,v)\|^2_{L^\infty_{t,x}}\right)\\
&+\|\tA(r,v)-1\|_{L^\infty_{t,x}}\|\Phi_t\|_{L^\infty L^{4/(4-s)}}\\
&\cdot\left(\|D^{s-2}(\tA^{-1}(r,v))\|_{L^\infty L^{4/(s-2)}}+\|D^{s-2}(\tA^{-2}(r,v))\|_{L^\infty L^{4/(s-2)}}\right)\\
&+\|\tA(r,v)-1\|_{L^\infty_{t,x}}\|D^{s-2}\Phi_t\|_{L^\infty L^2}\\
&\cdot\left(\|\tA^{-1}(r,v)\|_{L^\infty_{t,x}}+\|\tA^{-1}(r,v)\|^2_{L^\infty_{t,x}}\right).
\endaligned\label{box-phit-hs-2}
\end{equation}

First, it is easily seen that  \eqref{ta-formula} implies
\begin{equation}
\|\tA^{-1}(r,v)\|_{L^{\infty}(\R^5)}\lesssim 1.
\label{ta-inv-li}
\end{equation}

Secondly, we work on the norms involving $\Phi_t$, for which the Sobolev embeddings \eqref{Sob-gen} yield
\begin{equation*}
\|\Phi_t\|_{L^\infty L^{4/(4-s)}}\lesssim \|\Phi_t\|_{L^\infty H^{s-2}}\lesssim \|\Phi_t\|_{L^\infty H^{2}},
\end{equation*}
since $3<s<4$. One also has
\begin{equation*}
\|D^{s-2}\Phi_t\|_{L^\infty L^2}\sim \|\Phi_t\|_{L^\infty \dot{H}^{s-2}}\lesssim \|\Phi_t\|_{L^\infty H^2}
\end{equation*} 
and hence, using \eqref{phitlil2} and \eqref{Phit-lih2}, we obtain
\begin{equation*}
\|\Phi_t\|_{L^\infty L^{4/(4-s)}}+\|D^{s-2}\Phi_t\|_{L^\infty L^2}\lesssim 1.
\end{equation*}

Next, we address the norms depending on $\tA$ and $\tA^{-1}$. With the help of \eqref{Lbnz-0}, \eqref{ta-inv-ta}, and \eqref{ta-inv-li}, we deduce
\begin{equation*}
\aligned
\|D^{s-2}(\tA^{-1}(r,v))\|_{L^\infty L^{4/(s-2)}}&+\|D^{s-2}(\tA^{-2}(r,v))\|_{L^\infty L^{4/(s-2)}}\\
&\lesssim \|D^{s-2}(\tA^{-1}(r,v))\|_{L^\infty L^{4/(s-2)}}\left(1+\|\tA^{-1}(r,v)\|_{L^\infty_{t,x}}\right)\\
&\lesssim \|D^{s-2}(\tA(r,v))\|_{L^\infty L^{4/(s-2)}}.
\endaligned
\end{equation*}
By applying the interpolation inequality \eqref{interpol-bd} and \eqref{nabla-delta-A}, we derive
\begin{equation*}
\|D^{s-2}(\tA(r,v))\|_{L^\infty L^{4/(s-2)}}\lesssim \|\nabla_x(\tA(r,v))\|^{4-s}_{L^\infty L^4}\|\Delta(\tA(r,v))\|^{s-3}_{L^\infty L^2}\lesssim 1.
\end{equation*}
In the end, if we rely on \eqref{decay-A-3} and \eqref{ta-inv-li}, then we also control the $L^\infty_{t,x}$ norms in \eqref{box-phit-hs-2} and thus \eqref{box-phit-hs-1} is proved.
\end{proof}

Following this, we can finish the argument for \eqref{phi-phit-hs} and, consequently, the proof of our main result by coming up with the expected Sobolev regularity for $\Phi$. 

\begin{prop}
Under the assumptions of Proposition \ref{prop-phit}, we have
\begin{equation}
\|\Phi\|_{L^\infty \dot{H}^{s}}\lesssim 1.
\label{phi-hs+1}
\end{equation}
\end{prop}
\begin{proof}
We start the argument by taking advantage of \eqref{phit-phitt-hs} and $3<s<4$ to infer that
\begin{equation}
\aligned
\|\Phi\|_{L^\infty \dot{H}^{s}}\sim \|\Delta\Phi\|_{L^\infty \dot{H}^{s-2}}&\leq \|\Phi_{tt}\|_{L^\infty \dot{H}^{s-2}}+\|\Box\Phi\|_{L^\infty \dot{H}^{s-2}}\\
&\lesssim 1+\|\Box\Phi\|_{L^\infty H^{2}}.
\endaligned
\label{phi-hs+1-v2}
\end{equation}
On one hand, the combination of \eqref{rll1}, \eqref{rgrt1}, and \eqref{decay-v-3} produces
\begin{equation}
\aligned
\|\Box&\Phi\|_{L^\infty L^{2}}\\&\lesssim \left\|\min\{v^2, |v|^3\}\right\|_{L^\infty L^2([0,T)\times \{r<1\})}+\left\|\frac{|\varphi_{\geq 1/2}|}{r^3}+\frac{|v|}{r^2}\right\|_{L^\infty L^2([0,T)\times \{r\geq1\})}\\
&\lesssim \left\|1\right\|_{L^\infty L^2([0,T)\times \{r<1\})}+\left\|\frac{1}{r^3}\right\|_{L^\infty L^2([0,T)\times \{r\geq1\})}\\
&\lesssim 1.
\endaligned
\label{box-phi-l2}
\end{equation}

On the other hand, we notice that
\begin{equation}
\|\Box\Phi\|_{L^\infty \dot{H}^{2}}\sim\|\Delta\Box\Phi\|_{L^\infty L^{2}}
\label{box-phi-h2}
\end{equation}
and, subsequently, a direct computation based on \eqref{Box-Phi} yields
\begin{equation}
\aligned
\Delta\Box\Phi=\,&\frac{1}{2}\int_{0}^{v}\left\{\left(\tA^{-1/2}+3\tA^{-5/2}\right)\Delta \tA\right\}\ dy\\
&-\frac{1}{4}\int_{0}^{v}\left\{\left(\tA^{-3/2}+15\tA^{-7/2}\right)\tA_r^2\right\}\ dy\\
&+\frac{1}{2}\left(\tA^{-1/2}(r,v)+3\tA^{-5/2}(r,v)\right)\left(\partial_r\{\tA(r,v)\}+\tA_r(r,v)\right)v_r\\
&+\left(\tA^{1/2}(r,v)-\tA^{-3/2}(r,v)\right)\Delta v\\
&+\frac{\varphi_{\geq 1/2}}{r^3}.
\endaligned
\label{delta-box-phi}
\end{equation}
In previous arguments, we already relied on the last term of the right-hand side having finite $L^\infty L^2$ norm. Moreover, one can observe easily that 
\[
\tA_r(r,v)=\partial_r\{\tA(r,v)\}-\tA_y(r,v)v_r
\]
and
\begin{equation*}
\left|\tA_y(r,v)\right|=\frac{|\sin (2u)|}{r}\lesssim \tA^{1/2}(r,v).
\end{equation*}
Thus, with the help of \eqref{decay-A-3}, \eqref{nabla-delta-v}, and \eqref{nabla-delta-A}, we obtain
\begin{equation*}
\aligned
&\left\|\frac{1}{2}\left(\tA^{-1/2}(r,v)+3\tA^{-5/2}(r,v)\right)\left(\partial_r\{\tA(r,v)\}+\tA_r(r,v)\right)v_r\right\|_{L^\infty L^2}\\
&\qquad\qquad\qquad\lesssim \left(\left\|\partial_r(\tA_r(r,v))\right\|_{L^\infty L^4}+\left\|\tA^{1/2}(r,v)\right\|_{L^\infty_{t,x}}\|v_r\|_{L^\infty L^4}\right)\|v_r\|_{L^\infty L^4}\\
&\qquad\qquad\qquad\lesssim 1
\endaligned
\end{equation*}
and
\begin{equation*}
\left\|\left(\tA^{1/2}(r,v)-\tA^{-3/2}(r,v)\right)\Delta v\right\|_{L^\infty L^2}\lesssim \left\|\tA^{1/2}(r,v)\right\|_{L^\infty_{t,x}}\|\Delta v\|_{L^\infty L^2}\lesssim 1.
\end{equation*}

Hence, we are left to investigate the two integral terms in \eqref{delta-box-phi}. For the second one, we apply the trivial bound $\tA\geq 1$, \eqref{ta-r}, \eqref{ta-r-g1}, and \eqref{ta-r-l1} to deduce
\begin{equation*}
\aligned
\left|\frac{1}{4}\int_{0}^{v}\left\{\left(\tA^{-3/2}+15\tA^{-7/2}\right)\tA_r^2\right\}\ dy\right|
&\lesssim \int_{0}^{|v|}\tA_r^2\ dy\\
&\lesssim \begin{cases}
\frac{|v|+|v|^3}{r^4}, \, &r\geq 1,\\
r^2|v|^9, \, &r< 1.
\end{cases}
\endaligned
\end{equation*}
By now factoring in \eqref{decay-v-3}, we derive
\begin{equation*}
\left\|\frac{1}{4}\int_{0}^{v}\left\{\left(\tA^{-3/2}+15\tA^{-7/2}\right)\tA_r^2\right\}\ dy\right\|_{L^\infty L^2}\lesssim 1.
\end{equation*}
In what concerns the first integral term, we use \eqref{ta-r} to calculate
\begin{equation*}
\aligned
\Delta\tA=\,&\frac{\sin2(ry+\varphi)\cdot \varphi_{rr}+ 2\cos2(ry+\varphi)\cdot(y+\varphi_r)^2}{r^2}\\
&-\frac{\sin2(ry+\varphi)\cdot(y+\varphi_r)}{r^3}.
\endaligned
\end{equation*}
If we reason as we did for \eqref{ta-r-g1} and \eqref{ta-r-l1}, then in this case, it follows that
\begin{equation*}
\left|\Delta\tA\right|\lesssim \begin{cases}
\frac{1+y^2}{r^2}, \, &r\geq 1,\\
y^4, \, &r< 1,
\end{cases}
\end{equation*}
and, consequently,
\begin{equation*}
\aligned
\left|\frac{1}{2}\int_{0}^{v}\left\{\left(\tA^{-1/2}+3\tA^{-5/2}\right)\Delta \tA\right\}\ dy\right|&\lesssim \int_{0}^{|v|}\left|\Delta \tA\right|\ dy\\
&\lesssim \begin{cases}
\frac{|v|+|v|^3}{r^2}, \, &r\geq 1,\\
|v|^5, \, &r< 1.
\end{cases}
\endaligned
\end{equation*}
By applying again \eqref{decay-v-3}, we infer that
\begin{equation*}
\left\|\frac{1}{2}\int_{0}^{v}\left\{\left(\tA^{-1/2}+3\tA^{-5/2}\right)\Delta \tA\right\}\ dy\right\|_{L^\infty L^2}\lesssim 1, 
\end{equation*}
which, jointly with the estimates for the other terms in \eqref{delta-box-phi}, leads to
\begin{equation*}
\|\Delta\Box\Phi\|_{L^\infty L^{2}}\lesssim 1.
\end{equation*}
When combined with \eqref{phi-hs+1-v2}, \eqref{box-phi-l2}, and \eqref{box-phi-h2}, this bound shows that \eqref{phi-hs+1} holds true and the proof is concluded.
\end{proof}

\begin{remark}
The statements and arguments for the results in this subsection mirror the ones obtained in the corresponding part of \cite{GGr-17}.
\end{remark}

%%%%%%%%%%%%%%%%%%%%%%%%%%%%%%%%%%%%%%%%%%%%%%%%%%%%

\section*{Appendix}

As discussed in the outline of the paper, we focus in this section on arguing that the hypothesis of Theorem \ref{main-th-v-2} (equivalent through \eqref{utov} to the one of Theorem \ref{main-th}) is enough to claim 
the Sobolev regularity of various expressions evaluated at $t=0$, which was assumed to be true in certain steps of the main argument.

First, we relied on the finiteness of the energy \eqref{tote-0}, for which a straightforward analysis using Sobolev embeddings, radial Sobolev estimates, and Hardy-type inequalities shows that it is valid if 
\[
(u_0,u_1)\in \left(\dot{H}^{3/2+\epsilon}\cap\dot{H}^{1-\epsilon}\right)(\R^2)\times L^2(\R^2),
\]
with $\epsilon>0$ being arbitrarily small. Later, in section \ref{sect-en}, it is easy to see that the reasoning goes through if $v(0)\in H^{3/2}(\R^4)$ (e.g., proof of \eqref{phir-l2}). However, according to the hypotheses of Theorem \ref{main-th} and Theorem \ref{main-th-v-2}, we have $(u_0, u_1)\in H^s\times H^{s-1}(\R^2)$ and $v(0)\in H^s(\R^4)$, respectively, with $s>3$.

Following this, we need to check that three more statements are true:
\begin{equation}
\|\Phi_t(0)\|_{\dot{H}^{1}( \R^4)}+\|\Phi_{tt}(0)\|_{L^2(\R^4)}\lesssim 1,
\label{phit0-h1}
\end{equation}
featured in the argument for Proposition \ref{prop-Phit-h1},
\begin{equation}
 \|\Phi_{tt}(0)\|_{\dot{H}^{1}( \R^4)}+\|\Phi_{ttt}(0)\|_{L^2(\R^4)}\lesssim 1,
\label{phitt0-h1}
\end{equation}
appearing in the proof of Proposition \ref{prop-Phitt-h1}, and
\begin{equation}
 \|\Phi_t(0)\|_{\dot{H}^{s-1}( \R^4)}+\|\Phi_{tt}(0)\|_{\dot{H}^{s-2}(\R^4)}\lesssim 1,
\label{phit0-hs}
\end{equation}
which shows up in Proposition \ref{prop-phit}.

We begin by recalling \eqref{Phit-final} and \eqref{Phitt-final}, i.e., 
\begin{equation}
\Phi_{t}=\tA^{1/2}(r,v)\,v_{t},\qquad
\Phi_{tt}=\tA^{1/2}(r,v)\,v_{tt}+ \tA^{-1/2}(r,v)\,\frac{\sin(2u)}{2r}\,v^2_{t}.
\label{phit-phitt}
\end{equation}
With the help of the latter, we calculate
\begin{equation}
\aligned
\Phi_{ttt}=\ &\tA^{1/2}(r,v)\,v_{ttt}+ \frac{3}{2}\tA^{-1/2}(r,v)\,\frac{\sin(2u)}{r}\,v_{tt}\,v_{t}\\
&+\tA^{-3/2}(r,v)\left(\cos(2u)-\,\frac{\sin^4(u)}{r^2}\right)v^3_{t}.
\endaligned
\label{phittt}
\end{equation}
Moreover, the hypothesis of Theorem \ref{main-th-v-2} on the initial data, i.e.,
\begin{equation}
v(0)\in H^s(\R^4), \qquad v_t(0)\in H^{s-1}(\R^4), \qquad s>3,
\label{v0-hs}
\end{equation}
ensures, when combined with the classical Sobolev embedding \eqref{Sob-classic}, 
\begin{equation}
v(0)\in H^{1,\infty}(\R^4), \qquad v_t(0)\in L^{\infty}(\R^4),
\label{dv0-li}
\end{equation}
Now, we can proceed to prove \eqref{phit0-h1}. 

\begin{customthm}{A.1}
Under the assumptions of Theorem \ref{main-th-v-2}, the estimate \eqref{phit0-h1} is valid.
\end{customthm}
\begin{proof}
Due to 
\[
\|\Phi_t(0)\|_{\dot{H}^{1}( \R^4)}\sim \|\partial_r\Phi_t(0)\|_{L^{2}( \R^4)}, 
\]
we use \eqref{phit-phitt} to derive 
\begin{equation*}
\partial_r\Phi_t= \tA^{1/2}(r,v)\,\partial_rv_{t}+\frac{1}{2}\tA^{-1/2}(r,v)\left(-\frac{2\sin^2(u)}{r^3}+\frac{\sin(2u)}{r^2}u_r\right)v_{t},
\end{equation*}
which is analyzed separately in the $\{r\leq 1\}$ and $\{r>1\}$ regions. For the former, we have
\begin{equation}
1\leq \tA(r,v)\lesssim 1+v^2
\label{ta-l1}
\end{equation}
and an expansion in Maclaurin series yields
\begin{equation}
\left|-\frac{2\sin^2(u)}{r^3}+\frac{\sin(2u)}{r^2}u_r\right|\lesssim |v_r||v| +rv^4.
\label{sin2u-r3-l1}
\end{equation}
If $r>1$, then
\begin{equation}
\tA(r,v)\sim 1
\label{ta-g1}
\end{equation}
and 
\begin{equation}
\left|-\frac{2\sin^2(u)}{r^3}+\frac{\sin(2u)}{r^2}u_r\right|\lesssim \frac{1}{r^3}+\frac{|u_r|}{r^2}\lesssim \frac{1}{r^3}+\frac{|v|+r|v_r|+|\varphi_r|}{r^2}.
\label{sin2u-r3-g1}
\end{equation}
Thus, by applying \eqref{v0-hs}, and \eqref{dv0-li}, we deduce
\begin{equation*}
\aligned
\|\partial_r\Phi_t(0)\|_{L^{2}( \R^4)}&\lesssim \left(1+\|v(0)\|_{L^{\infty}( \R^4)}\right)\|\partial_rv_t(0)\|_{L^{2}( \R^4)}\\
&\quad+\big(\|v_r(0)\|_{L^{\infty}( \R^4)}\|v(0)\|_{L^{\infty}( \R^4)}+\|v(0)\|^4_{L^{\infty}( \R^4)}\\
&\qquad\ +1+\|v_r(0)\|_{L^{\infty}( \R^4)}\big)\|v_t(0)\|_{L^{2}( \R^4)}\\
&\lesssim 1.
\endaligned
\end{equation*}

It follows that we need to argue for
\begin{equation}
\|\Phi_{tt}(0)\|_{L^2(\R^4)}\lesssim 1,
\label{Phitt-0}
\end{equation}
and, in relation to this goal, we observe that
\begin{equation}
\frac{|\sin(2u)|}{r}\lesssim\begin{cases}
|v|, \quad r\leq 1,\\
\frac{1}{r}, \quad \ \, r >1.
\end{cases}
\label{sin2u}
\end{equation}
Together with the bounds employed above in connection to $\Phi_t(0)$, this estimate implies
\begin{equation}
\aligned
\|\Phi_{tt}(0)\|_{L^2(\R^4)}&\lesssim  \left(1+\|v(0)\|_{L^{\infty}( \R^4)}\right)\|v_{tt}(0)\|_{L^{2}( \R^4)}\\
&\quad+ \left(1+\|v(0)\|_{L^{\infty}( \R^4)}\right)\|v_{t}(0)\|_{L^{\infty}( \R^4)}\|v_{t}(0)\|_{L^{2}( \R^4)}\\
&\lesssim \|v_{tt}(0)\|_{L^{2}( \R^4)}+1\\
&\leq \|\Box v(0)\|_{L^{2}( \R^4)}+\|\Delta v(0)\|_{L^{2}( \R^4)}+1\\
&\lesssim \|\Box v(0)\|_{L^{2}( \R^4)}+1.
\endaligned
\label{ptt-Dv}
\end{equation}

At a first glance, one could say now that \eqref{Phitt-0} is proved by simply invoking \eqref{vtt-dv}. However, a careful inspection reveals that \eqref{phit0-h1} is used implicitly in the proof of \eqref{vtt-dv} and thus this attempt is circular. Instead, we rely on asymptotics developed in the argument for \eqref{vtt-dv}, as these are independent of \eqref{phit0-h1}. If we analyze separately each term on the right-hand side of \eqref{main-v} when $t=0$, we obtain:
\begin{equation*}
\left\| \frac{1}{r}\Delta_2 \varphi\right\|_{L^2(\R^4)}+\left\|\frac{1}{r^2}\varphi_{>1} v(0)\right\|_{L^2(\R^4)} \lesssim \left\|\frac{1}{r}\right\|_{L^2( \{1\leq r\leq 2\})}+\|v(0)\|_{L^2(\R^4)}\lesssim 1,
\end{equation*}
\begin{equation*}
\aligned
&\left\|\varphi_{<1}\left(\frac{1}{r}N(r, rv(0), \nabla(rv)(0))+\frac{1}{r^2}v(0)\right)\right\|_{L^2(\R^4)}\\
&\qquad\lesssim\left\| |\varphi_{<1}|\left(|v(0)|^3+|v(0)|^5+|v(0)||\nabla v(0)|^2+rv^4(0)|v_r(0)|\right)\right\|_{L^2(\R^4)}\\
&\qquad\lesssim \bigg(\|v(0)\|^2_{L^{\infty}( \R^4)}+\|v(0)\|^4_{L^{\infty}( \R^4)}+\|\nabla v(0)\|^2_{L^{\infty}( \R^4)}\\
&\qquad\qquad+\|v(0)\|^3_{L^{\infty}( \R^4)}\|\nabla v(0)\|_{L^{\infty}( \R^4)}\bigg)\|v(0)\|_{L^{2}( \R^4)}\\
&\qquad\lesssim 1,
\endaligned
\end{equation*}
\begin{equation*}
\aligned
&\left\|\frac{1}{r}\varphi_{>1}N\left(r, rv+\varphi, \nabla(rv+\varphi)\right)\big|_{t=0}\right\|_{L^2(\R^4)}\\
&\qquad\quad\lesssim \left\||\varphi_{>1}|\left(\frac{1}{r^3}+\frac{|\nabla v(0)|^2}{r}+\frac{|v(0)|^2}{r^3}\right)\right\|_{L^2(\R^4)}\\&\qquad\quad\lesssim 1+\|\nabla v(0)\|_{L^{\infty}( \R^4)}\|\nabla v(0)\|_{L^{2}( \R^4)}+\|v(0)\|_{L^{\infty}( \R^4)}\|v(0)\|_{L^{2}( \R^4)}\\
&\qquad\quad\lesssim 1,
\endaligned
\end{equation*}
which jointly show that
\begin{equation}
\|\Box v(0)\|_{L^{2}( \R^4)}\lesssim 1.
\label{boxv-0} 
\end{equation}
Hence, due to \eqref{ptt-Dv}, one has that \eqref{Phitt-0} holds true and the proof of this proposition is concluded.
\end{proof}

Following this result, we can build upon its analysis and show that \eqref{phitt0-h1} is valid.

\begin{customthm}{A.2}
Under the assumptions of Theorem \ref{main-th-v-2}, the estimate \eqref{phitt0-h1} holds true.
\end{customthm}
\begin{proof}
First, we reduce the proof of \eqref{phitt0-h1} to showing that
\begin{equation}
\|\nabla\Box v(0)\|_{L^{2}( \R^4)}\lesssim 1.
\label{nabla-boxv-0}
\end{equation}
In order to estimate the $L^2$ norm for $\Phi_{ttt}(0)$, we work with \eqref{phittt} to infer that
\begin{equation*}
\aligned
\|\Phi_{ttt}(0)\|&_{L^{2}( \R^4)}\\
&\lesssim \|\tA^{1/2}(r,v(0))\|_{L^{\infty}( \R^4)}\|v_{ttt}(0)\|_{L^{2}( \R^4)}\\
&\quad + \left\|\frac{\sin(2u(0))}{r}\right\|_{L^{\infty}( \R^4)}\|v(0)\|_{L^{\infty}( \R^4)}\|v_{tt}(0)\|_{L^{2}( \R^4)}\\
&\quad +\left\|\cos(2u(0))-\frac{\sin^4(u(0))}{r^2}\right\|_{L^{\infty}( \R^4)}\|v_{t}(0)\|^2_{L^{\infty}( \R^4)}\|v_{t}(0)\|_{L^{2}( \R^4)}.
\endaligned
\end{equation*}
Using \eqref{ta-l1}-\eqref{sin2u-r3-g1}), we deduce
\begin{equation}
\|\tA^{1/2}(r,v(0))\|_{L^{\infty}( \R^4)}+\left\|\frac{\sin(2u(0))}{r}\right\|_{L^{\infty}( \R^4)}\lesssim 1+\|v(0)\|_{L^{\infty}( \R^4)}
\label{ta-sin2u}
\end{equation}
and 
\begin{equation}
\|v_{tt}(0)\|_{L^{2}( \R^4)}\lesssim 1.
\label{vtt-0}
\end{equation}
Furthermore, one easily notices that
\begin{equation*}
\left|\cos(2u)-\frac{\sin^4(u)}{r^2}\right|\lesssim\begin{cases}
1+r^2v^4, \quad r\leq 1,\\
1, \quad \qquad \quad\, r >1,
\end{cases}
\end{equation*}
and, subsequently, 
\begin{equation*}
\left\|\cos(2u(0))-\frac{\sin^4(u(0))}{r^2}\right\|_{L^{\infty}( \R^4)}\lesssim 1+\|v(0)\|^4_{L^{\infty}( \R^4)}.
\end{equation*}
Hence, by taking advantage of these estimates, \eqref{v0-hs}, and \eqref{dv0-li}, we derive
\begin{equation}
\aligned
\|\Phi_{ttt}(0)\|_{L^{2}( \R^4)}&\lesssim \|v_{ttt}(0)\|_{L^{2}( \R^4)}+1\\
&\leq\|\Delta v_{t}(0)\|_{L^{2}( \R^4)}+\|\partial_t \Box v(0)\|_{L^{2}( \R^4)}+1\\
&\lesssim \|\partial_t \Box v(0)\|_{L^{2}( \R^4)}+1.
\endaligned
\label{dt-boxv-0}
\end{equation}

In what concerns the $\dot{H}^1$ norm of $\Phi_{tt}(0)$, due to
\[
\|\Phi_{tt}(0)\|_{\dot{H}^{1}( \R^4)}\sim \|\partial_r\Phi_{tt}(0)\|_{L^{2}( \R^4)}, 
\]
we apply \eqref{phit-phitt} to calculate
\begin{equation*}
\aligned
\partial_r\Phi_{tt}&=\tA^{1/2}(r,v)\,\partial_rv_{tt}\\
&+\frac{1}{2}\left(-\frac{2\sin^2(u)}{r^3}+\frac{\sin(2u)}{r^2}u_r\right)\left(\tA^{-1/2}(r,v)v_{tt}-\tA^{-3/2}(r,v)\frac{\sin(2u)}{2r}v_t^2\right)\\
&+\tA^{-1/2}(r,v)\left(-\frac{\sin(2u)}{2r^2}+\frac{\cos(2u)}{r}u_r\right)v^2_{t}\\
&+\tA^{-1/2}(r,v)\,\frac{\sin(2u)}{r}\,\partial_r v_t\,v_{t}.
\endaligned
\end{equation*}
If we rely on \eqref{sin2u-r3-l1} and \eqref{sin2u-r3-g1}, then we obtain
\begin{equation*}
\aligned
&\left\|-\frac{2\sin^2(u(0))}{r^3}+\frac{\sin(2u(0))}{r^2}u_r(0)\right\|_{L^\infty(\R^4)}\\
&\qquad\qquad \lesssim\|v_r(0)\|_{L^{\infty}( \R^4)}\|v(0)\|_{L^{\infty}( \R^4)}+\|v(0)\|^4_{L^{\infty}( \R^4)}+1+\|v_r(0)\|_{L^{\infty}( \R^4)}.
\endaligned
\end{equation*}
In a similar manner, we first infer that 
\begin{equation*}
\left|-\frac{\sin(2u)}{2r^2}+\frac{\cos(2u)}{r}u_r\right|\lesssim\begin{cases}
|v_r|+r|v|^3, \qquad\qquad\quad\, r\leq 1,\\
\frac{1}{r^2}+\frac{|v|+r|v_r|+|\varphi_r|}{r}, \quad  \quad\, r >1,
\end{cases}
\end{equation*}
and, thus, 
\begin{equation*}
\left\|-\frac{\sin(2u)}{2r^2}+\frac{\cos(2u)}{r}u_r\right\|_{L^{\infty}( \R^4)}\lesssim \|v_r(0)\|_{L^{\infty}( \R^4)}+\|v(0)\|^3_{L^{\infty}( \R^4)}+1.
\end{equation*}
By bringing in the mix \eqref{v0-hs}, \eqref{dv0-li}, \eqref{ta-sin2u}, and \eqref{vtt-0}, we deduce
\begin{equation*}
\aligned
&\|\partial_r\Phi_{tt}(0)\|_{L^{2}( \R^4)}\\
&\qquad\lesssim (1+\|v(0)\|_{L^{\infty}( \R^4)})\|\partial_rv_{tt}(0)\|_{L^2(\R^4)}\\
&\quad\qquad+\left(\|v_r(0)\|_{L^{\infty}( \R^4)}\|v(0)\|_{L^{\infty}( \R^4)}+\|v(0)\|^4_{L^{\infty}( \R^4)}+1+\|v_r(0)\|_{L^{\infty}( \R^4)}\right)\\
&\quad\qquad\quad \cdot\left\{\|v_{tt}(0)\|_{L^2(\R^4)}+(1+\|v(0)\|_{L^{\infty}( \R^4)})\|v_t(0)\|_{L^{\infty}( \R^4)}\|v_t(0)\|_{L^2( \R^4)}\right\}\\
&\quad\qquad+\left(\|v_r(0)\|_{L^{\infty}( \R^4)}+\|v(0)\|^3_{L^{\infty}( \R^4)}+1\right)\|v_t(0)\|_{L^{\infty}( \R^4)}\|v_t(0)\|_{L^2( \R^4)}\\
&\quad\qquad+(1+\|v(0)\|_{L^{\infty}( \R^4)})\|v_t(0)\|_{L^{\infty}( \R^4)}\|\partial_rv_{t}(0)\|_{L^2(\R^4)}\\
&\qquad\lesssim \|\partial_rv_{tt}(0)\|_{L^2(\R^4)}+1\\
&\qquad\leq \|\partial_r\Delta v(0)\|_{L^2(\R^4)}+\|\partial_r\Box v(0)\|_{L^2(\R^4)}+1\\
&\qquad\lesssim \|\partial_r\Box v(0)\|_{L^2(\R^4)}+1.
\endaligned
\end{equation*}
Jointly with \eqref{dt-boxv-0}, this inequality shows that \eqref{phitt0-h1} follows if \eqref{nabla-boxv-0} is proved.

In proving \eqref{nabla-boxv-0}, we proceed by estimating the gradient $\nabla_{t,r}=(\partial_t, \partial_r)$ evaluated at $t=0$ for every single term on the right-hand side of \eqref{main-v}. This is done by observing that the resulting expressions share a generic core with the corresponding ones analyzed in connection to \eqref{boxv-0}. Therefore, we can strictly work on the slight differences featured in this new setting. First, a direct argument yields
\begin{equation*}
\aligned
&\left\|\nabla_{t,r}\left( \frac{1}{r}\Delta_2 \varphi\right)\right\|_{L^2(\R^4)}+\left\|\nabla_{t,r}\left(\frac{1}{r^2}\varphi_{>1} v\right)(0)\right\|_{L^2(\R^4)}\\ 
&\qquad\qquad\qquad\lesssim 1+\|v(0)\|_{L^2(\R^4)}+\|\nabla_{t,r}v(0)\|_{L^2(\R^4)}\\
&\qquad\qquad\qquad\lesssim 1.
\endaligned
\end{equation*}

Next, for terms involving the cutoff $\varphi_{<1}$, the gradient for expressions having the generic profile $\tilde{N}(rv)v^k$ can be easily estimated based on \eqref{dni-li}. Indeed, we deal with terms like
\[
\tilde{N}(rv)v^{k-1}\nabla_{t,r}v, \qquad \tilde{N}'(rv)v^{k}r\nabla_{t,r}v, \qquad \tilde{N}'(rv)v^{k+1}.
\]
and, by comparison to the analysis for \eqref{boxv-0}, $v(0)$ is replaced by $\nabla_{t,r}v(0)$ or an extra factor of $r\nabla_{t,r}v(0)$ or $v(0)$ appears. In the former scenario, the gradient is bounded in the same $L^p$ space as we bounded $v(0)$. For the latter, both extra factors are estimated in  $L^\infty(\R^4)$ using \eqref{dv0-li}, since the presence of $\varphi_{<1}$ forces $r\leq 1$. We can write similar proofs for the terms $N_3(rv)v(v_t^2-v_r^2)$ and $N_4(rv)rv^4v_r$, with slight adjustments when the gradient is applied to the derivative terms. In this situation, we are faced with estimating
\[
N_3(rv(0))v(0)(v_t(0)\nabla_{t,r}v_t(0)-v_r(0)\nabla_{t,r}v_r(0))
\]
and
\[
N_2(rv(0))rv^4(0)\nabla_{t,r}v_r(0),
\]
and all factors are bounded in $L^\infty(\R^4)$, with the exception of the second order derivatives, which are placed in $L^2(\R^4)$. We use  \eqref{vtt-0} for $v_{tt}(0)$, whereas 
\[
\|\partial_tv_r(0)\|_{L^2(\R^4)}=\|\partial_rv_t(0))\|_{L^2(\R^4)}\sim \|v_t(0))\|_{\dot{H}^1(\R^4)}\lesssim 1.
\]
To control $v_{rr}(0)$, we rely on \eqref{rad-Sob-2} and \eqref{dv0-li} to deduce
\begin{equation*}
\aligned
\|v_{rr}(0))\|_{L^2(\R^4)}&\lesssim \|\Delta v(0))\|_{L^2(\R^4)}+\left\|\frac{v_{r}(0)}{r}\right\|_{L^2(\R^4)}\\ &\lesssim \|v(0))\|_{\dot{H}^2(\R^4)}+\left\|\frac{1}{(1+r^{3/2})r}\right\|_{L^2(\R^4)}\\
&\lesssim 1,
\endaligned
\end{equation*}
which concludes the discussion of terms localized by $\varphi_{<1}$.

In what concerns the gradient for terms involving $N(r,rv+\varphi, \nabla(rv+\varphi))$, we argue that the analysis is virtually equivalent to the one above, with one exception. The differentiation introduces extra factors of $r$, which are potentially dangerous due to the presence of $\varphi_{>1}$. However, we ask the careful reader to check that, in fact, the structure of $N(r,rv+\varphi, \nabla(rv+\varphi))$ contains sufficient negative powers of $r$ to counteract this issue.
\end{proof}

\noindent\textbf{Remark A.3.}
\textit{These two propositions coincide in their statement with the corresponding results proved for the Skyrme model. Moreover, the two sets of arguments are roughly equivalent, with little modifications due to differences in the formulas for $\Phi_{tt}$ and $\Phi_{ttt}$ between this paper and \cite{GGr-17}.}

The last result of this appendix certifies that \eqref{phit0-hs} holds true.

\begin{customthm}{A.4}
Under the assumptions of Theorem \ref{main-th-v-2}, the estimate \eqref{phit0-hs} is valid.
\end{customthm}
\begin{proof}
We start by addressing the $\dot{H}^{s-1}(\R^4)$ norm and we apply the fractional Leibniz estimate \eqref{Lbnz-0} in the context of \eqref{phit-phitt} to derive
\begin{equation*}
\aligned
\|\Phi_t(0)\|_{\dot{H}^{s-1}( \R^4)}\lesssim\, &\|\tA^{1/2}(r,v(0))\|_{\dot{H}^{s-1}( \R^4)}\|v_t(0)\|_{L^\infty( \R^4)}\\ &+ \|\tA^{1/2}(r,v(0))\|_{L^\infty( \R^4)}\|v_t(0)\|_{\dot{H}^{s-1}( \R^4)}.
\endaligned
\end{equation*}
Next, by virtue of \eqref{v0-hs}, \eqref{dv0-li}, and \eqref{ta-sin2u}, we infer that
\begin{equation}
\|\Phi_t(0)\|_{\dot{H}^{s-1}( \R^4)}\lesssim \|\tA^{1/2}(r,v(0))\|_{\dot{H}^{s-1}( \R^4)}+1.
\label{phiths-1}
\end{equation}
Following this, we use the Moser inequality \eqref{Moser} for the $C^\infty$ function
\[
F:\R\to \R, \qquad F(x)=(1+x^2)^{1/2}-1,
\]
to obtain
\begin{equation}
\aligned
 &\|\tA^{1/2}(r,v(0))\|_{\dot{H}^{s-1}( \R^4)}\\
&\qquad\qquad\lesssim  \|\tA^{1/2}(r,v(0))-1\|_{H^{s-1}( \R^4)}\\
&\qquad\qquad\lesssim \gamma\left(\left\|\frac{\sin(u(0))}{r}\right\|_{L^\infty(\R^4)}\right)\left\|\frac{\sin(u(0))}{r}\right\|_{H^{s-1}(\R^4)}\\
 &\qquad\qquad\lesssim \gamma\left(\left\|\frac{\sin(u(0))}{r}\right\|_{L^\infty(\R^4)}\right)\bigg(\left\|\frac{\sin(u(0))}{r}\right\|_{H^{s-1}(\{1\leq r\leq 2\})}\\
 &\qquad\qquad\qquad\qquad\qquad\qquad\qquad\qquad+\left\|\frac{\sin(rv(0))}{r}\right\|_{H^{s-1}(\R^4)}\bigg)\\
 &\qquad\qquad\lesssim \left\|\frac{\sin(rv(0))}{r}\right\|_{H^{s-1}(\R^4)}+1,
\endaligned
\label{phiths-2}
\end{equation}
where we can motivate the last line by an argument identical to the one leading to \eqref{ta-sin2u}.

Subsequently, we rely on expansions in Maclaurin series to deduce
\begin{equation*}
\aligned
\frac{\sin(rv(0))}{r}=\ &v(0)\sum_{k\geq 0}\frac{(-1)^{3k}(rv(0))^{6k}}{(6k+1)!}+r^2v^3(0)\sum_{k\geq 0}\frac{(-1)^{3k+1}(rv(0))^{6k}}{(6k+3)!}\\
&+r^4v^5(0)\sum_{k\geq 0}\frac{(-1)^{3k+2}(rv(0))^{6k}}{(6k+5)!}\\
=\ &\left(1+H_1((rv(0))^{6}\right)v(0)+\left(\frac{-1}{6}+H_2((rv(0))^{6})\right)r^2v^3(0)\\
&+\left(\frac{1}{120}+H_3((rv(0))^{6})\right)r^4v^5(0),
\endaligned
\end{equation*}
where the function $H_i\in C^\infty(\R;\R)$ satisfies $H_i(0)=0$, for all $1\leq i\leq 3$. Given that $s>3$, $H^{s-1}(\R^4)$ is an algebra and, by also taking advantage of \eqref{v0-hs} and \eqref{Moser}, we derive
\begin{equation}
\aligned
&\left\|\frac{\sin(rv(0))}{r}\right\|_{H^{s-1}(\R^4)}\\
&\quad\qquad\lesssim \left(1+\gamma(\|r^6v^6(0)\|_{L^\infty( \R^4)})\,\|r^6v^6(0)\|_{H^{s-1}( \R^4)}\right)\\
&\quad\qquad\quad\left(\|v(0)\|_{H^{s-1}( \R^4)}+ \|r^2v^3(0)\|_{H^{s-1}( \R^4)}+ \|r^4v^5(0)\|_{H^{s-1}( \R^4)}\right)\\
&\quad\qquad\lesssim \left(1+\gamma(\|r^6v^6(0)\|_{L^\infty( \R^4)})\,\|r^6v^6(0)\|_{H^{s-1}( \R^4)}\right)\\
&\quad\qquad\quad\left(1+ \|r^2v^3(0)\|_{H^{s-1}( \R^4)}+ \|r^4v^5(0)\|_{H^{s-1}( \R^4)}\right).
\endaligned
\label{phiths-3}
\end{equation}
Next, we can obviously work with $v(0)\in H^1({\R^5})$ and, according to \eqref{rad-Sob-2} and \eqref{dv0-li}, we infer that
\begin{equation}
|v(0)|\lesssim \frac{1}{1+r^{3/2}}
\label{v0-r32}
\end{equation}
and, consequently,
\begin{equation*}
\|r^2v^3(0)\|_{L^2( \R^4)}+\|r^4v^5(0)\|_{L^2( \R^4)}+\|r^6v^6(0)\|_{L^\infty\cap L^2( \R^4)}\lesssim 1.
\end{equation*}
Thus, based on \eqref{phiths-1}-\eqref{phiths-3}, we obtain
\begin{equation}
\|\Phi_t(0)\|_{\dot{H}^{s-1}( \R^5)}\lesssim 1,
\label{phit0-hs-1}
\end{equation}
if we show that
\begin{equation*}
\|r^2v^3(0)\|_{\dot{H}^{s-1}( \R^4)}+ \|r^4v^5(0)\|_{\dot{H}^{s-1}( \R^4)}+ \|r^6v^6(0)\|_{\dot{H}^{s-1}( \R^4)}\lesssim 1.
\end{equation*}
We claim that there are clear similarities between the ways one should analyze the above three norms. This is why we present here only the argument for the first one and ask the diligent reader to fill in the details for the other two. Since $s>3$, we infer that  $D^{s-1}(r^2)=0$, which, jointly with a more involved Kato-Ponce type inequality (see Theorem 1.2 in \cite{Li-16}), \eqref{v0-hs}, and \eqref{dv0-li}, yields
\begin{equation*}
\aligned
\|r^2v^3(0)\|_{\dot{H}^{s-1}( \R^4)}&\sim \|D^{s-1}\left(r^2v^3(0)\right)\|_{L^{2}( \R^4)}\\
&\lesssim\|rv(0)\|^2_{L^\infty( \R^4)}\|v(0)\|_{\dot{H}^{s-1}( \R^4)}\\ &\quad+ \|rv^2(0)\|_{L^\infty( \R^4)}\|v(0)\|_{\dot{H}^{s-2}( \R^4)}\\
&\quad+\|v(0)\|^2_{L^\infty( \R^4)}\|v(0)\|_{\dot{H}^{s-3}( \R^4)}\\
&\lesssim \|rv(0)\|^2_{L^\infty( \R^4)}+\|rv^2(0)\|_{L^\infty( \R^4)}+1.
\endaligned
\end{equation*}
However, \eqref{v0-r32} easily implies
\[
\|rv(0)\|^2_{L^\infty( \R^4)}+\|rv^2(0)\|_{L^\infty( \R^4)}\lesssim 1
\]
and the proof of \eqref{phit0-hs-1} is concluded.

Following this, we are left to investigate the $\dot{H}^{s-2}( \R^5)$ norm in \eqref{phit0-hs} and, for this purpose, we begin by analyzing the second term on the right-hand side of \eqref{phit-phitt}. If we apply \eqref{Lbnz-0}, \eqref{v0-hs}, \eqref{dv0-li}, \eqref{ta-sin2u}, and the trivial estimate
\[
0<\tA^{-1/2}(r,v(0))\leq 1,
\]
then we deduce 
\begin{equation*}
\aligned
&\left\|\tA^{-1/2}(r,v(0))\,\frac{\sin(2u(0))}{2r}\,v^2_{t}(0)\right\|_{\dot{H}^{s-2}( \R^4)}\\
&\qquad\lesssim \left\|\tA^{-1/2}(r,v(0))\right\|_{\dot{H}^{s-2}( \R^4)}\left\|\frac{\sin(2u(0))}{r}\right\|_{L^\infty(\R^4)}\left\|v_{t}(0)\right\|^2_{L^\infty( \R^4)}\\
&\qquad\quad+\left\|\tA^{-1/2}(r,v(0))\right\|_{L^\infty( \R^4)}\left\|\frac{\sin(2u(0))}{r}\right\|_{\dot{H}^{s-2}( \R^4)}\left\|v_{t}(0)\right\|^2_{L^\infty( \R^4)}\\
&\qquad\quad+\left\|\tA^{-1/2}(r,v(0))\right\|_{L^\infty( \R^4)}\left\|\frac{\sin(2u(0))}{r}\right\|_{L^\infty( \R^4)}\left\|v_{t}(0)\right\|_{L^{\infty}( \R^4)}\left\|v_{t}(0)\right\|_{\dot{H}^{s-2}( \R^4)}\\
&\qquad\lesssim \left\|\tA^{-1/2}(r,v(0))\right\|_{\dot{H}^{s-2}( \R^4)}+\left\|\frac{\sin(2u(0))}{r}\right\|_{\dot{H}^{s-2}( \R^4)}+1.
\endaligned
\end{equation*}
However, by arguing similarly to the way we derived \eqref{phit0-hs-1}, we can infer that 
\[
\left\|\tA^{-1/2}(r,v(0))\right\|_{H^{s-1}( \R^4)}+\left\|\frac{\sin(2u(0))}{r}\right\|_{H^{s-1}( \R^4)}\lesssim 1
\]
and, subsequently, 
\[
\left\|\tA^{-1/2}(r,v(0))\,\frac{\sin(2u(0))}{2r}\,v^2_{t}(0)\right\|_{\dot{H}^{s-2}( \R^4)}\lesssim 1.
\]
Finally, we focus on the first term on the right-hand side of \eqref{phit-phitt}, and we use \eqref{Lbnz-0}, \eqref{Sob-gen}, \eqref{ta-sin2u}, \eqref{vtt-0}, and the analysis deriving \eqref{phit0-hs-1} to obtain
\begin{equation*}
\aligned
\left\|\tA^{1/2}(r,v(0))\,v_{tt}(0)\right\|_{\dot{H}^{s-2}( \R^5)}&\lesssim \left\|\tA^{1/2}(r,v(0))\right\|_{\dot{H}^{s-2, 4}( \R^4)}\left\|v_{tt}(0)\right\|_{L^4( \R^4)}\\
&\quad+\left\|\tA^{1/2}(r,v(0))\right\|_{L^\infty( \R^4)}\left\|v_{tt}(0)\right\|_{\dot{H}^{s-2}( \R^4)}\\
&\lesssim \left\|\tA^{1/2}(r,v(0))\right\|_{\dot{H}^{s-1}( \R^4)}\left\|v_{tt}(0)\right\|_{H^{s-2}( \R^4)}\\
&\quad+\left\|v_{tt}(0)\right\|_{\dot{H}^{s-2}( \R^4)}\\
&\lesssim \left\|v_{tt}(0)\right\|_{\dot{H}^{s-2}( \R^4)}+1.
\endaligned
\end{equation*}
Now, due to \eqref{v0-hs}, we deduce
\[
\aligned
 \left\|v_{tt}(0)\right\|_{\dot{H}^{s-2}( \R^4)}&\lesssim \left\|\Delta v(0)\right\|_{\dot{H}^{s-2}( \R^4)}+\left\|\Box v(0)\right\|_{\dot{H}^{s-2}( \R^4)}\\
 &\lesssim \left\|v(0)\right\|_{\dot{H}^{s}( \R^4)}+\left\|\Box v(0)\right\|_{\dot{H}^{s-2}( \R^4)}\\
 &\lesssim \left\|\Box v(0)\right\|_{\dot{H}^{s-2}( \R^4)}+1
 \endaligned
\]
and we claim that, following the framework in the analysis for the $\dot{H}^{s-1}( \R^4)$ norm, one also derives
\[
\left\|\Box v(0)\right\|_{\dot{H}^{s-2}( \R^4)}\lesssim 1.
\]
We let the avid reader verify all the details. In the end, by combining the last four estimates, we infer that
\begin{equation}
\|\Phi_{tt}(0)\|_{\dot{H}^{s-2}( \R^4)}\lesssim 1
\label{phit0-hs-2}
\end{equation}
and the proof of \eqref{phit0-hs} is finished.
\end{proof}

\noindent\textbf{Remark A.5.}
\textit{This result matches the statement of its counterpart in \cite{GGr-17}. However, the argument here is considerably more involved than the one written for the Skyrme model. There, one has 
\begin{equation*}
|v(0)|\lesssim \frac{1}{1+r^{2}}
\end{equation*}
instead of \eqref{v0-r32} and, subsequently,
\begin{equation*}
\|r^2v^2(0)\|_{L^\infty\cap L^2( \R^5)}\lesssim 1.
\end{equation*}
This leads to a relevant simplification of the proof for
\[
\left\|\frac{\sin(rv(0))}{r}\right\|_{H^{s-1}(\R^5)}\lesssim 1,
\]
in the sense that we can work with a much simpler expansion in Maclaurin series.}

\bibliographystyle{amsplain}
\bibliography{anwb-recent}

\end{document}